\documentclass[reqno]{amsart}
\usepackage{amsmath,amsbsy,amsfonts}
\usepackage{color}

\usepackage[utf8]{inputenc}
\usepackage{graphicx}
\usepackage{amsmath,latexsym,amssymb}
\newtheorem{theorem}{Theorem}
\newtheorem{proposition}{Proposition}[section]

\newtheorem{definition}[proposition]{Definition}
\newtheorem{lemma}[proposition]{Lemma}
\newtheorem{remark}[proposition]{Remark}

\newcommand{\dd}{\mathrm{d}}

\numberwithin{equation}{section}


\usepackage{color}
\numberwithin{equation}{section}

\title[Critical Ambrosetti-Prodi type problems for  fractional $p$-Laplacian ]{ Critical concave convex  Ambrosetti-Prodi type problems for fractional $p$-Laplacian}

\author[H. P. Bueno]{H. P. Bueno}\thanks{First author takes part in the project 422806/2018-8 by CNPq/Brazil}
\author[E.Huerto Caqui]{E. Huerto Caqui}\thanks{Second author was supported by CAPES/Brazil}
\author[O. H. Miyagaki]{O. H. Miyagaki }\thanks{Third author was supported by Grant 2019/24901-3 by S\~ao Paulo Research Foundation (FAPESP) and Grant 307061/2018-3 by CNPq/Brazil.}
\author[F.R.Pereira] {F. R. Pereira}
\address[H. P. Bueno]{ Department of Matematics, Universidade Federal de Minas Gerais, 31270-901 - Belo Horizonte-MG, Brazil}
\email{hamilton@mat.ufmg.br}
\address[E. Huerto Caqui]{ Department of Matematics, Universidade Federal de Minas Gerais, 31270-901 - Belo Horizonte-MG, Brazil}
\email{analisis\_11@hotmail.com}
\address[O. H. Miyagaki]{Department of Mathematics, Universidade Federal de S\~ao Carlos, 13565-905 - S\~ao Carlos-SP, Brazil}
\email{olimpio@ufscar.br, ohmiyagaki@gmail.com}
\address[F. R. Pereira]{Department of Mathematics, Universidade Federal de Juiz de Fora, 36036-330 - Juiz de Fora-MG, Brazil}
\email{fabio.pereira@ufjf.edu.br}


\begin{document} 

\begin{abstract}
In this paper we consider a class of critical  concave convex Ambrosetti-Prodi type problems 
involving the fractional  $p$-Laplacian operator. By applying the Linking Theorem and the Mountain Pass Theorem as well, 
the interaction of the  nonlinearities with the first eigenvalue of the fractional  $p$-Laplacian will be used to prove existence of multiple solutions. 
\end{abstract}
\maketitle

\section{Introduction}
Let $\Omega\subset\mathbb{R}^N$ be a bounded smooth domain, $s\in (0,1)$ and $N>sp$.
In this paper we investigate the existence of multiple solutions for the following nonlocal problem 
\begin{align}\label{principal}
\left\{
\begin{array}{rlll}
(-\Delta)^{s}_{p} u&=& -\lambda\vert u\vert^{q-2}u+a\vert u\vert^{p-2}u+b(u^{+})^{p^*_{s}-1} &\textrm{in}\;\;\Omega,\\
u&=&0 &\textrm{in}\;\;\mathbb{R}^N\setminus\Omega,
\end{array}
\right.
\end{align}
where $(-\Delta)^{s}_p$ is the fractional $p$-Laplacian operator defined by
$$(-\Delta)^{s}_pu(x)=2\displaystyle\lim_{\epsilon\to 0}
 \displaystyle\int_{\mathbb{R}^N\setminus B(x,\epsilon)}\frac{\vert u(x)-u(y)\vert^{p-2}(u(x)-u(y))}{\vert x-y\vert^{N+sp}}\dd y, 
 \; \; x\in\mathbb{R}^{N}.$$
In \eqref{principal}, $\lambda>0$ is a parameter, $a, b>0$ are real constants, $1<q<p$, $p_s^*=pN/(N-sp)$ is the critical Sobolev exponent for $(-\Delta)^{s}_{p}$ and
$u^{+}:=\max\{0,u\}$ denotes the positive part of $u$, while the negative part of $u$ will be denoted $u^{-}=\min\{0,u\}.$
Consequently, $u=u^+ + u^-$. 

The term $(u^+)$ appears for the first time in the classical paper by Ruf and Srikanth \cite{Ruf}. But our problem is also related to two classical local problems, namely, the Ambrosetti-Prodi and the Brezis-Nirenberg problems. 
On these subjects see the excellent  books \cite{Fucik}  and \cite{willem}, respectively.

In 1972, Ambrosetti and Prodi  \cite{Am}  considered the Dirichlet boundary value problem
$$ -\Delta u =g(u) + f(x) \  \mbox{in} \  \Omega,\quad u=0 \  \mbox{on}\  \partial \Omega,\leqno{(AP)}$$
where $\Omega \subset \mathbb{R}^N$  is a bounded smooth domain, $-\Delta$ denotes the Laplacian operator, $f:\mathbb{R}^N\rightarrow \mathbb{R}$ is a $C^1$ function, $g:\mathbb{R} \rightarrow \mathbb{R}$ is $C^2$, convex and satisfies
$$0<g_{-}=\lim_{t \rightarrow -\infty}g'(t)< \lambda_1 < g_{+}=\lim_{t \rightarrow +\infty} g'(t)< \lambda_2,$$
with $\lambda_1 $ and $\lambda_2$ denoting the first and second eigenvalues of $(-\Delta, H^{1}_{0}(\Omega))$. 
They proved the existence of a $C^1$ manifold $M$ in $C^{0,\alpha}(\overline{\Omega})$, 
which splits the space into two open sets $O_0$ and
$O_2$ with the following properties
\begin{itemize}
\item [$(i)$] if $f \in O_0$, problem ($AP$) has no solution;
\item [$(ii)$] if $f \in M$, problem ($AP$) has exactly one solution;
\item [$(iii)$] if $f \in O_2$, problem ($AP$) has exactly two solutions.
\end{itemize}

Many authors have extended this result in different ways and we would like to apologize if we omit some important contributions, 
but  we cite, e.g., the papers \cite{ Amann, Aubin, Bere, berger, Calanchi, dancer, deFiguei, Je, Ruf, STZ} and references therein. 
All these results show the role of the interaction between $g_{\pm}$ and the eigenvalues of $(-\Delta, H^{1}_{0}(\Omega))$.

On the other hand, in 1983, Brezis and Nirenberg \cite{BN}  studied the Dirichlet boundary problem
$$ -\Delta u = a u + |u|^{ 2^{*}-2} u \  \mbox{in} \  \Omega,\quad u=0 \  \mbox{on}\  \partial \Omega,\leqno{(BN)}$$
where $\Omega \subset \mathbb{R}^N$ is a bounded smooth domain, $a > 0$ and $2^{*} =2N/(N-2)\ (N\geq 3)$ is the Sobolev critical exponent. 
As before, denoting by $\lambda_j \; (j=1,2, \dots)$ the eigenvalues of $-\Delta$, the authors proved that there exists $a_0>0$ such that 
\begin{itemize}
\item [$(i)$] if $ a_0 < a< \lambda_1$, $ N=3$ and $\Omega=B_1(0)$, problem $(BN)$ has at least one positive solution;
\item [$(ii)$] if $a < \lambda_1 $  and $N\geq 4,$ problem $(BN)$ has at least one positive solution.
\end{itemize}
Among many works extending or complementing the above result for both local and nonlocal operators, we mention e.g. 
\cite{Ambr,AM,barrios,Brasco,CALANCHI,Ch,KHIDDI,mosconi,pere,serval}. But we would like to highlight Capozzi, Fortunato and 
Palmieri \cite{Capozzi}, where the authors proved that problem $(BN)$ has at least one nontrivial solution for all $a>0$  if $N\geq 5$  and for all $a \neq \lambda_j$ if $N=4$.  

In the interesting work \cite{paiva1}, de Paiva and Presoto established a multiplicity result for the Dirichlet boundary problem 
$$ -\Delta u = -\lambda |u |^{q-2}u + au + ( u^+)^{ p-1}  \  \mbox{in} \  \Omega,\quad u=0 \  \mbox{on}\  \partial \Omega,\leqno{(PP)}$$
where $\Omega \subset \mathbb{R}^N$ is a bounded smooth domain, $a,\lambda  > 0$ and $ 1< q< 2< p \leq 2^*$. (See also \cite{paiva} for the subcritical case and \cite{HUANG,mont} 
for related problems.) The main goal of the present paper is to prove the result
obtained in \cite{paiva1} for the fractional $p$-Laplacian operator extending the results obtained in \cite{O} and in \cite{BHM}.

Acknowledgments. The authors wish to thank an anonymous referee that helped us to simplify the presentation of the article.
\section{Notations and preliminary stuff}
For any measurable function $u:\mathbb{R}^N\to\mathbb{R}$ the Gagliardo seminorm is defined by
$$
[u]_{s,p}:=\Big(\int_{\mathbb{R}^{2N}}\frac{|u(x)-u(y)|^p}{|x-y|^{N+sp}} dxdy\Big)^{1/p}.
$$
We consider the fractional Sobolev space
$$
W^{s,p}(\mathbb{R}^N)=\{u\in L^p(\mathbb{R}^N):\,\, [u]_{s,p}<\infty\}
$$
endowed with the norm $\|u\|_{W^{s,p}}=(\|u\|^p_{L^p}+[u]_{s,p}^p)^{1/p}$. Since solutions should be equal to zero outside of $\Omega$, it is natural to consider the closed linear subspace given by 
$$X_p^s=\{u\in W^{s,p}(\mathbb{R}^{N})\, :\, u = 0 \textrm{ in } \mathbb{R}^N\setminus\Omega \}$$
equivalently renormed by setting $\| \cdot \|_{X_p^s}= [\cdot]_{s,p} $, which is a uniformly convex Banach space. The imbedding $X_p^s \hookrightarrow L^r(\Omega)$ is continuous for $r \in [1, p^*_s ]$ and compact for  $r \in [1, p^*_s )$.

We define, for all $u,v\in X_p^s$, the operator $A:X_p^s\to (X_{p}^{s})^*$ by
\begin{equation*}
A(u)\cdot v=\int_{\mathbb{R}^{2N}}\frac{\vert u(x)-u(y)\vert^{p-2}(u(x)-u(y))(v(x)-v(y))}{\vert x-y\vert^{N+sp}}\dd x\dd y.
\end{equation*}

\begin{definition}
We say that $u\in X_p^s$ is a weak solution to \eqref{principal} if
$$A(u)\cdot v=-\lambda\int_{\Omega}\vert u\vert^{q-2}uv\dd x+a\int_{\Omega}\vert u\vert^{p-2}uv\dd x+b\int_{\Omega}(u^+)^{p_s^*-1}v\dd x,$$ 
for all $v\in X_p^s.$
\end{definition}

Since the action functional
$I_{\lambda,s}:X_p^s\to\mathbb{R}$ is given by 
\begin{equation*}
I_{\lambda,s}(u)=\displaystyle\frac{1}{p}\|u\|_{X_p^s}^p+\frac{\lambda}{q}\int_{\Omega}\vert u\vert^{q}\dd x-\frac{a}{p}\int_{\Omega}\vert u\vert^p\dd x-\frac{b}{p_s^*}\int_\Omega (u^+)^{p_s^*}\dd x,
\end{equation*}
we have
$$I_{\lambda,s}'(u)\cdot v= A(u)\cdot v+\lambda\int_{\Omega}\vert u\vert^{q-2}uv\dd x-a\int_{\Omega}\vert u\vert^{p-2}uv\dd x-b\int_{\Omega}\!(u^+)^{p_s^*-1}v\dd x,$$
thus implying that critical points of $I_{\lambda,s}$ are weak solutions of \eqref{principal}.
 
In the local problem studied by de Paiva and Presoto \cite{paiva1}, the driving operator used is the standard Laplacian on the Sobolev space $H_0^1 (\Omega)$. In our work, in order to obtain two solutions of opposite constant sign for problem \eqref{principal} as in \cite{paiva1}, we apply the Mountain Pass Theorem  to the positive and negative parts of the functional $I_{\lambda, s}$. In view of the essential differences in the functional setting ($H_0^1 (\Omega)$ versus $X_p^s$ and $-\Delta$ versus  $(-\Delta)^s_p$), it was necessary to obtain a result (see Theorem \ref{t1}) relating the local minimizers in different spaces ($ C_{\delta}^0 (\overline{\Omega})$ versus $ X_p^s $) to obtain the geometrical conditions of the mountain pass,  whose proof was inspired by the works  \cite{ni} and \cite{Saoudi} for semilinear and quasilinear problems, respectively. (The space $C_{\delta}^0(\overline{\Omega})$ will be defined in Section \ref{minimizer}).
 
A third solution to the problem \eqref{principal} was obtained via the Linking Theorem (see \cite{rabinowitz})
adapting arguments found in Miyagaki, Motreanu and Pereira \cite{O}, de Paiva and Presoto \cite{paiva1}, but mainly in de Figueiredo and Yang \cite{deFiguei}. The presence of the fractional $p$-Laplacian makes impossible the application of approximate eigenfunctions, 
which were used in \cite{paiva1}. Furthermore, since an explicit formula for minimizers of the  best constant of the Sobolev 
immersion ${X_p^s}\hookrightarrow L^{p_s^*}(\Omega)$ is not available, another difficulties arise. Partial solutions were obtained by Chen, Mosconi and Squassina \cite{Chen}, Mosconi, Perera, Squassina and Yang \cite{mosconi} and also by Brasco, Mosconi and Squassina \cite{Brasco}. 

In order to obtain the geometric conditions of the Linking Theorem, we define
\begin{equation*}
\lambda^{*}=\inf\left\{\|u\|_{X_p^s}^{p}\; : \; u\in W, \ \|u\|_{L^{p}(\Omega)}^{p}=1\right\},
\end{equation*}
where $$W=\left\{ u\in X_p^s\; :\; A(\varphi_1)\cdot u =0\right\},$$
with $\varphi_1$ the first eigenfunction of $(-\Delta)_p^s$, positive and $L^p$-normalized associated with the first eigenvalue $\lambda_1$.  

Following ideas of Alves, Carrião and Miyagaki \cite{alves} and Anane and Tsouli \cite{Anane} (see also \cite{Capozzi}), it is not difficult to obtain the next result, see \cite{BHM} for details.
\begin{proposition}
$\lambda_1<\lambda^{*}$.
\end{proposition}

We are now in a position to establish the main result of this paper.
\begin{theorem}\label{t0}
Suppose $\lambda_1<a<\lambda^{*}$, $b>0$, $1<q<p$ and assume that one of the following conditions hold,
\begin{enumerate}
\item[$(i)$] $N>sp^2$ and $1<p\leq\dfrac{2N}{N+s}$, 
\item[$(ii)$] $N>sp((p-1)^2+p)$ and $p>\dfrac{2N}{N+s}$.
\end{enumerate}
Then problem \eqref{principal} has at least three nontrivial solutions if $\lambda>0$ is small enough. 
\end{theorem}

\begin{remark} In Theorem \ref{t0}, we consider only two of the six possibilities below, \begin{enumerate}
\item [$(a)$] $1<p\leq\dfrac{2N}{N+s}\quad$ $\textrm{ and }$ $\quad\left\{ \begin{array}{rc}
		(1)& N>sp^2,  \\
		(2)& N=sp^2, \\
		(3)& sp<N<sp^2, 
		\end{array} \right.$
\item [$(b)$] $\left[p>\dfrac{2N}{N+s},\  N>sp\big( (p-1)^2+p\big)\right]\quad$ $\textrm{ and }$ $\quad\left\{ \begin{array}{rc}
		(1)& N>sp^2,  \\
		(2)& N=sp^2, \\
		(3)& sp<N<sp^2,
		\end{array} \right.$
	\end{enumerate}
since it is not difficult to verify that the situations $(a)-(2)$, $(a)-(3)$, $(b)-(2)$ and $(b)-(3)$ are incompatible. 
\end{remark}

\section{$C_{\delta}^0$ versus $W^{s,p}$ minimization for polynomial growth}\label{minimizer}
The main result of this section is a local minimization equivalence for functionals defined in the fractional Sobolev space $X_p^s$ with polynomial growth nonlinearity, following ideas developed by Barrios, Colorado, de Pablo and Sanchéz \cite{barrios}, Giacomoni, Prashanth and Sreenadh \cite{giacomoni} and Iannizzotto,  Mosconi and Squassina \cite{iannizzoto}. The result we prove is more general than those found in \cite{barrios} and  \cite{iannizzoto}, since we allow $p>1$.

We start showing a regularization result that will be useful in the proof of Theorem \ref{t1}. Its proof is similar to that of \cite[Lemma 3.1]{BHM}. 
\begin{proposition}\label{CPL2_23.1}
Suppose $\vert g(t)\vert\leq C(1+\vert t\vert^{q-1})$, for some  $1\leq q\leq p^*_s$ and $\,C>0$. Let $(v_\epsilon)_{\epsilon\in(0,1)}\subseteq {X_p^s}$ be a bounded family of solutions  in  ${X_p^s}$ to the problem
\begin{align}\label{CPL2_23.2}
\left\{\begin{array}{rlll}
(-\Delta)^{s}_{p} u&=&\left(\displaystyle\frac{1}{1-\xi_{\epsilon}}\right)g(u) &\textrm{in}\;\;\Omega,\\
u&=&0 &\textrm{in}\;\;\mathbb{R}^N\setminus\Omega,
\end{array}\right.
\end{align}
with $\xi_\epsilon\leq 0$. Then  
$\displaystyle{\sup_{\epsilon\in(0,1)}}\|v_\epsilon\|_{L^{\infty}(\Omega)}<\infty.$
\end{proposition}
\begin{proof}
For $0<k\in\mathbb{N}$, we define
$$T_{k}(s)= \left\{ \begin{array}{rl}
s+k, &\textrm{if}\quad s\leq -k,\\
0, &\textrm{if}\quad -k<s<k,\\
s-k,&\textrm{if}\quad s\geq k
\end{array} \right.$$ 
and
$$\Omega_{k}=\{x\in\Omega\,:\; \vert v_{\epsilon}(x)\vert\geq k\}.$$ 
Observe that $T_{k}(v_{\epsilon})\in X_p^s$ and $\|T_{k}(v_{\epsilon})\|_{X_p^s}^p\leq C^p\|v_{\epsilon}\|_{X_p^s}^p<\infty$ for a constant $C>0$. 
	
Taking $T_{k}(v_{\epsilon})$ as a test-function, we obtain
\begin{align*}
A( v_{\epsilon})\cdot T_{k}(v_{\epsilon})=&\int_{\Omega} \left(\displaystyle\frac{1}{1-\xi_{\epsilon}}\right) g(v_{\epsilon})T_{k}(v_{\epsilon})\dd x\\
\leq&\int_{\Omega_k}C\vert T_{k}(v_{\epsilon})\vert \dd x+C\int_{\Omega_k}\vert v_\epsilon\vert^{p^*_s-1} \vert T_{k}(v_{\epsilon})\vert \dd x.
\end{align*} 
	
Now consider $1<\theta_3<\theta_2<p$,  $p<\dfrac{\theta_2}{\theta_3}+1$ and
$(p_s^*-1)\theta_1<p_s^*$ such that $\theta_1 ^{-1}+\theta_2^{-1}+\theta_3^{-1}=1$. Therefore, by applying Hölder's inequality, we obtain
\begin{equation}\label{CPL2_23.4}
A(v_{\epsilon})\cdot T_{k}(v_{\epsilon})\leq C\left(\int_{\Omega}\vert T_{k}(v_{\epsilon})\vert^{\theta_2} \dd x\right)^{1/\theta_2}\vert\Omega_{k}\vert^{1/\theta_3}.
\end{equation}
	
Denote $$T(x,y)=\dfrac{\vert  v_{\epsilon}(x)-v_{\epsilon}(y)\vert^{p-2}(v_{\epsilon}(x)-
v_{\epsilon}(y))(T_{k}(v_{\epsilon})(x)-T_{k}(v_{\epsilon})(y))}{\vert x-y\vert^{N+sp}}.$$

Noting that the following inequality holds
$$|s-t|^{p-2}(s-t)(T_k(s)-T_k(t))\geq |T_k(s)-T_k(t)|^{p}, \;\; \mbox{for all} \; \;s,t \in \mathbb{R},$$
since both $T_k (s)$ and $s-T_k (s)$ are non decreasing functions, we obtain
$$T(x,y) \geq \frac{\vert  T_{k}(v_{\epsilon})(x)-T_{k}(v_{\epsilon})(y)\vert^{p}}{\vert x-y\vert^{N+sp}}.$$

Therefore, we have the estimate
\begin{align*}
A(v_{\epsilon})\cdot T_{k}(v_{\epsilon})\geq 
\int\limits_{\mathbb{R}^{2N}}\frac{\vert  T_{k}(v_{\epsilon})(x)-
T_{k}(v_{\epsilon})(y)\vert^{p}}{\vert x-y\vert^{N+sp}}\dd x\dd y = \|T_{k}(v_{\epsilon})\|_{X_p^s}^{p}.
\end{align*}

It follows from the continuous immersion $X_p^s\hookrightarrow L^{\theta_2}(\Omega)$ that (for a constant $C_1>0$)
\begin{equation}\label{CPL2_23.5}
C_1 \left(\int_{\Omega}\vert T_{k}(v_{\epsilon})\vert^{\theta_2} \dd x\right)^{p/\theta_2}\leq A(v_{\epsilon})\cdot T_{k}(v_{\epsilon}).
\end{equation}

So, \eqref{CPL2_23.4} and \eqref{CPL2_23.5} guarantee the existence of $C>0$ such that 
\begin{equation*}
\int_{\Omega}\vert T_{k}(v_{\epsilon})\vert^{\theta_2} \dd x\leq C\vert\Omega_{k}\vert^{\theta_2/\theta_3(p-1)}=C\vert\Omega_{k}\vert^{\beta/\beta-1},
\end{equation*}
where $\beta=\dfrac{\theta_2}{\theta_2-\theta_3(p-1)}>1$, the last inequality being a consequence of $p<\dfrac{\theta_2}{\theta_3}+1$.

Since, for all $s\in\mathbb{R}$, we have $\vert T_k(s)\vert=(\vert s\vert-k)(1-\chi_{[-k,k]}(s))$, we conclude that, if $0<k<h\in\mathbb{N}$, then $\Omega_{h}\subset\Omega_k$. Therefore
\begin{align*}
\int_{\Omega}\vert T_{k}(v_{\epsilon})\vert^{\theta_2} \dd x =&\int_{\Omega_k}(\vert v_{\epsilon}\vert-k)^{\theta_2}\dd x\geq \int_{\Omega_h}(\vert v_{\epsilon}\vert-k)^{\theta_2}\dd x\geq (h-k)^{\theta_2}\vert\Omega_h\vert.
\end{align*}

Defining, for $0<k\in\mathbb{N}$, $$\phi(k)=\vert\Omega_k\vert,$$  it follows 
\begin{equation*}
\phi(h)\leq C(h-k)^{-\theta_2}\phi(k)^{\beta/(\beta-1)}, \quad 0<k<h\in\mathbb{N}.
\end{equation*}

Considering the sequence $(k_n)$ defined by $k_0=0$ and $k_n=k_{n-1}+d/2^n$, where  $d=2^{\beta}C^{1/\theta_2}\vert\Omega\vert^{1/(p-1)\theta_2}$, we have $0\leq\phi(k_n)\leq\phi(0)/(2^{nr(\beta-1)})$ for all $n\in\mathbb{N}$. Thus $ {\lim_{n\to\infty}}\phi(k_n)=0$.

Since $\phi(k_n)\geq\phi(d)$ implies $\phi(d)=0$, we have $\vert v_\epsilon(x)\vert\leq d$ a.e. in $\Omega$, for all $\epsilon\in(0,1)$. We are done.
$\hfill\Box$\end{proof}\vspace*{.5cm}

We recall the definitions of the spaces $C_{\delta}^0(\overline{\Omega})$ and $C_{\delta}^{0,\alpha}(\overline{\Omega})$. For this, let $\delta:\overline{\Omega}\to\mathbb{R}^+$ be given by
$\delta(x)=\textup{dist}(x,\mathbb{R}^N\setminus\Omega)$. Then, for $0<\alpha<1$, we have
\begin{align*}
C_{\delta}^{0}(\overline{\Omega})&=\left\{u\in C^{0}(\overline{\Omega})
\;:\;\frac{u}{\delta^{s}}\;\textrm{ has a continuous extension to } \overline{\Omega}\right\}\\
C_{\delta}^{0,\alpha}(\overline{\Omega})&=\left\{u\in C^0(\overline{\Omega})\;:\;\frac{u}{\delta^{s}}\;\textrm{ has a $\alpha$-Hölder extension to } \overline{\Omega}\right\}
\end{align*}
with the respective norms
$$\|u\|_{0,\delta}=\left\|\displaystyle\frac{u}{\delta^{s}}\right\|_{L^{\infty}(\Omega)} \ \mbox{and}\  \|u\|_{\alpha,\delta}=\|u\|_{0,\delta}+\displaystyle{\sup_{x,y\in\overline{\Omega},\,x\neq y.}}\frac{\left\vert u(x)/\delta(x)^{s}- u(y)/\delta(y)^{s}\right\vert}{\vert x-y\vert^{\alpha}}. $$

Let us consider the Dirichlet problem
\begin{align}\label{a.2}
\left\{\begin{array}{rlll}
(-\Delta)^{s}_{p} u&=&f(u) &\textrm{in}\ \Omega,\\
u&=&0 &\textrm{in}\ \mathbb{R}^N\setminus\Omega,
\end{array}
\right.
\end{align}
where $\Omega\subset\mathbb{R}^N$ ($N>1$) is a bounded, smooth domain, $s\in(0,1)$, $p>1$ and $f\in L^{\infty}(\Omega)$.

The next two results can be found in Iannizzotto, Mosconi and Squassina \cite{iannizzoto2}, Theorems 1.1 and 4.4, respectively. They will play a major role in the proof of Theorem \ref{t1}.

\begin{proposition}\label{a.3}
There exist $\alpha\in(0,s]$ and $C_{\Omega}>0$ depending only on $N$, $p$, $s$, with $C_{\Omega}$ also depending on $\Omega$, such that, for all weak solution $u\in X^s_p$ of \eqref{a.2}, $u\in C^\alpha(\overline{\Omega})$ and
$$\|u\|_{C^{\alpha}(\overline{\Omega})} \leq C_{\Omega}\|f\|_{L^{\infty}(\Omega)}^{\frac{1}{p-1}}.$$
\end{proposition} 
\begin{proposition}\label{a.4}
Let $u\in X^s_p$  satisfies $\left|(-\Delta)_{p}^{s} u\right| \leq K$ weakly in $\Omega$ for some $K>0$. Then 
$$|u| \leq\left(C_{\Omega} K\right)^{\frac{1}{p-1}} \delta^{s}\quad a.e.\,\textrm{ in }\Omega,$$ 
for some $C_{\Omega}=C(N, p, s, \Omega)$. 
\end{proposition}

The proof of the next result is similar to that of \cite[Theorem 1]{BHM}. We emphasize that, only in this section, $\delta$ represents the function defined by $\delta(x)=\textup{dist}(x,\partial\Omega)$. \vspace*{.2cm}

\begin{theorem}\label{t1}Suppose that $g\in C(\Omega)$ satisfies 
\begin{equation}\label{CPL2_24.0}
\vert g(t)\vert\leq C(1+\vert t\vert^{q-1}),\, \textrm{ for some }\,1\leq q\leq p^*_s\,\textrm{ and }\,C>0
\end{equation}
and consider the functional $\Phi:X_p^s\to\mathbb{R}$ defined by $$\Phi(u)=\dfrac{1}{p}\|u\|_{X_p^s}^p-\displaystyle\int_{\Omega} G(u)\dd x,$$ where $G(t)=\int_{0}^{t}g(s)\dd s$.

If $0$ is a local minimum of $\Phi$ in $C_{\delta}^0(\overline{\Omega})$, that is, there exists $r_1>0$ such that
\begin{equation*}
\Phi(0)\leq\Phi(z),\;\forall\;z\in X_p^s\cap C_{\delta}^0(\overline{\Omega}),\;\|z\|_{0,\delta}\leq r_1,
\end{equation*}
then $0$ is a local minimum of $\Phi$ in $X_p^s$, that is, there exists $r_2>0$ such that $$\Phi(0)\leq\Phi(z),\;\forall\; z\in X_p^s,\;\|z\|_{X_p^s}\leq r_2.$$
\end{theorem}

\begin{proof}
Let us consider initially the \textit{subcritical case}  $q<p_s^*$. By contradiction, denoting $\bar{B}_{\epsilon}=\{z\in {X_p^s}\;:\;\|z\|_{X_p^s}\leq\epsilon\}$, let us suppose that, for any  $\epsilon>0$, there exists $u_\epsilon\in \bar{B}_{\epsilon}$ such that
\begin{equation}\label{CPL2_24.1}
\Phi(u_\epsilon)<\Phi(0).
\end{equation} 

Since $\Phi:\bar{B}_{\epsilon}\to\mathbb{R}$ is weakly lower semicontinuous, 
there exists $v_\epsilon\in \bar{B}_{\epsilon}$ such that $\displaystyle\inf_{u\in \bar{B}_{\epsilon}}\Phi(u)=\Phi(v_\epsilon)$. It follows from \eqref{CPL2_24.1} that $$\Phi(v_\epsilon)=\inf_{u\in\bar{B}_{\epsilon}}\Phi(u)\leq\Phi(u_\epsilon)<\Phi(0).$$	
We will show that 
\begin{equation*}
v_{\epsilon}\rightarrow 0\textrm{ in } C_\delta^0(\overline{\Omega}) \textrm{ as } \epsilon\rightarrow 0,
\end{equation*}
since this implies, for $r_1>0$, 
the existence of $ z \in C_\delta^0(\overline{\Omega})$, such that $\|z\|_{0,\delta}<r_1$ and 
$\Phi(z)<\Phi(0)$, contradicting our hypothesis.

Since $v_\epsilon$ is a critical point of $\Phi$ in $X^s_p$, by Lagrange multipliers it is not difficult to verify 
\begin{equation}\label{multipliers}
\Phi'(v_{\epsilon})= \xi_{\epsilon}A(v_{\epsilon})
\end{equation}
implies $\xi_{\epsilon}\leq 0$.

Thus, it follows from \eqref{multipliers} that $v_{\epsilon}$ satisfies 
$$
\left\{
\begin{array}{rclll}
 (-\Delta)_{p}^s v_{\epsilon} &=& \left(\displaystyle\frac{1}{1-\xi_{\epsilon}}\right)g(v_{\epsilon}) =: g^{\epsilon}(v_{\epsilon})  &\textrm{ in } &\Omega,\\
v_{\epsilon}&=&0  &\textrm{ in } &\mathbb{R}^N\backslash \Omega,
\end{array} 
\right.  \leqno{(Pv)}
$$
If $ \|v_{\epsilon}\|_{X_p^s} \leq\epsilon<1$, Proposition \ref{CPL2_23.1} shows the existence of a constant $C_1>0$, not depending on $\epsilon$, such that
\begin{equation}\label{C2P_24.11.0}
\|v_{\epsilon}\|_{L^{\infty}(\Omega)} \leq C_1.
\end{equation}

Since $\xi_{\epsilon}\leq 0$,  \eqref{CPL2_24.0} and \eqref{C2P_24.11.0} imply that 
\begin{equation*}
\|g^{\epsilon}(v_{\epsilon})\|_{L^{\infty}(0,1)}\leq C_2
\end{equation*}
for some constant $C_2>0$.

Proposition \ref{a.3} yields $\|v_{\epsilon}\|_{C^{0,\beta}(\overline{\Omega})} \leq C_3$, for $0<\beta\leq s$ and a constant $C_3$ not depending on $\epsilon$. Now, it follows from Arzelà-Ascoli theorem the existence of a \textit{sequence} $(v_\epsilon)$ such that $v_{\epsilon}\rightarrow 0$ uniformly as $\epsilon\rightarrow 0$. Passing to a subsequence, we can suppose that $v_\epsilon\to 0$ a.e. in $\Omega$ and, therefore, $v_{\epsilon}\to 0$ uniformly in $\overline{\Omega}$. But now follows from Proposition \ref{a.4} that 
\begin{equation*}
\|v_{\epsilon}\|_{0,\delta}=\biggl\| \frac{v_{\epsilon}}{\delta^{s}} \biggl\|_{L^{\infty}(\Omega)} \leq  C \sup_{x\in (0,1)} |g^{\epsilon}(v_{\epsilon}(x))|
\end{equation*}
for a constant $C>0$. The proof of the subcritical case is complete.

We now consider the critical case $q=p_s^*$. As before, we argument by contradiction. For this, we define $g_k, G_k :\mathbb{R}\to\mathbb{R}$ by
$$g_k(s)=g(t_k(s))\quad \textrm{ and }\quad G_k(t)=\int_{0}^tg_k(s)\dd s,$$ 
with $t_k$ given by 
$$t_{k}(s)= \left\{ \begin{array}{rl}
-k, & \textrm{if}\quad s\leq-k \\	
s, &\textrm{if}\quad -k<s<k \\
k, &\textrm{if}\quad s\geq k.
\end{array} \right.$$

Considering $\Phi_{k}\in C^1({X_p^s},\mathbb{R})$ given by
$$\Phi_k(u)=\frac{\|u\|_{X_p^s}^p}{p}-\int_{\Omega}G_k(t)\dd t,$$
it follows $\Phi_k(u)\to\Phi(u)$ as $k\to\infty$. Thus, for any $\epsilon\in(0,1)$, there exists $k_{\epsilon}\geq 1$ such that $\Phi_{k_\epsilon}(w_\epsilon)<\Phi(0)$
and the subcritical growth of $g_k$ guarantees the existence of $u_\epsilon\in \bar{B}_{\epsilon}$ such that $$\Phi_{k_\epsilon}(u_\epsilon)=\displaystyle{\inf_{u\in \bar{B}_\epsilon}}\Phi_{k_\epsilon}(u)\leq\Phi_{k_\epsilon}(w_\epsilon)<\Phi(0).$$ 

As in the subcritical case, we find $\xi_\epsilon\leq 0$ such that $u_\epsilon$ is a weak solution to the problem  $(Pv)$ with  $u_\epsilon$ instead of  $v_\epsilon.$

From definition of $g_k,$ and since $\|u_\epsilon\|_{X_p^s}\leq\epsilon<1$, by applying Proposition \ref{CPL2_23.2} we obtain
\begin{equation*}
\|g_{k_\epsilon}^{\epsilon}(u_{\epsilon})\|_{L^{\infty}(\Omega)}\leq C_2
\end{equation*}
for a constant $C_2>0$. It follows from  Proposition \ref{a.3} that $\|u_{\epsilon}\|_{C^{0,\beta}(\overline{\Omega})} \leq C_3$, for $0<\beta\leq s$, the constant $C_3$ not depending on $\epsilon$. By applying Arzelà-Ascoli theorem, the conclusion is now obtained as in the subcritical case. 
$\hfill\Box$\end{proof}

\begin{remark}\label{strict}
If $0$ a strict local minimum in $C_{\delta}^{0}(\overline{\Omega})$, it follows that $0$ is also a strict local minimum in $X_p^s$.
\end{remark}

\section{Positive and Negative solutions}\label{Positivenegative}
Most of the results in this section are standard. Therefore, our presentation will be only schematic. 

We denote 
\begin{equation}\label{C3N_9}
S_{p,s}=\inf\left\lbrace\dfrac{\|u\|_{X_p^s}^p}{\left(\displaystyle\int_\Omega\vert u\vert^{p_s^*}\dd x\right)^{\frac{p}{p_s^*}}}\;;\;u\in {X_p^s}, u\neq 0\right\rbrace
\end{equation}
the best constant of the immersion ${X_p^s}\hookrightarrow L^{p_s^*}(\Omega)$, see \cite{mosconi}. 

If we take care of the operator $A$, the proof of the next two results is similar to that exposed in \cite{paiva1}.
\begin{lemma}
If $ a> \lambda_1$, $b>0$, $1<q<p$ and $\lambda>0$, then any $(PS)$-sequence of	$I_{\lambda,s}$ is bounded in $X_p^s.$
\end{lemma}
\begin{lemma}\label{C3P_3}
If $a> \lambda_1, \; b>0$, $1<q<p$ and $\lambda>0$, then $I_{\lambda,s}$ satisfies the $(PS)$-condition at any level $C$ such that $$C<\dfrac{s}{N}b^\frac{sp-N}{sp}S_{p,s}^\frac{N}{sp}.$$
\end{lemma}
We now consider the positive part of the functional $I_{\lambda,s}$. That is, $I_{\lambda,s}^+:{X_p^s}\to \mathbb{R}$ given by $$I_{\lambda,s}^+(u)=\frac{1}{p}\|u\|_{X_p^s}^p+\frac{\lambda}{q}\int_{\Omega} \vert u^+\vert^q\dd x-\frac{a}{p}\int_{\Omega}\vert u^+\vert^p\dd x-\frac{b}{p^*_s}\int_{\Omega} (u^{+})^{p^*_s}\dd x.$$
Of course, $I_{\lambda,s}^+\in  C^1({X_p^s},\mathbb{R})$ and it holds, for all $u,h\in {X_p^s},$
\begin{align*}
(I_{\lambda,s}^+)'(u)\cdot h= A(u)\cdot h+\lambda\int_{\Omega}\!\vert u^+\vert^{q-1}h\dd x
-a\int_{\Omega}\!\vert u^+\vert^{p-1}h\dd x-b\int_{\Omega} (u^+)^{p^*_s-1}h\dd x.
\end{align*}

Furthermore, critical points of $I_{\lambda,s}^+$ are weak solutions to the problem
\begin{equation*}
\left\{
\begin{array}{rcll}
(-\Delta)^{s}_{p} u&=&-\lambda\vert u^+\vert^{q-1}+ a\vert u^{+}\vert^{p-1}+b(u^+)^{p^*_s-1} &\textrm{in }\ \Omega,\\
u&=&0 &\textrm{in }\ \mathbb{R}^{N}\setminus \Omega,
\end{array}
\right.
\end{equation*}
where $a,b,\lambda >0$, $1<q<p$ and $u^+= \max\{ u,0\}$.

We now recall the following elementary inequality, which has a straightforward proof.
\begin{lemma}\label{C5P_33}
For all $u:\mathbb{R}^{N}\to\mathbb{R}$, $p>1$ and $x,y\in\mathbb{R}^{N}$ it holds
\begin{enumerate}
\item[$(i)$]$\vert u^+(x)-u^+(y)\vert^p\leq\vert u(x)-u(y)\vert^{p-2} \big(u(x)-u(y)\big)\big(u^+(x)-u^+(y)\big),$
\item[$(ii)$]$\vert u^-(x)-u^-(y)\vert^p\leq\vert u(x)-u(y)\vert^{p-2} \big(u(x)-u(y)\big)\big(u^{-}(x)-u^{-}(y)\big).$
\end{enumerate}
\end{lemma}\vspace*{.5cm}

If $u$ is a critical point of $I_{\lambda,s}^+$, then $(I_{\lambda,s}^+)'(u)\cdot h=0$ for all $h\in {X_p^s}$. Taking $h=u^-$, it follows from  
Lemma \ref{C5P_33} that
\begin{align*}
0=(I_{\lambda,s}^+)'(u)\cdot u^-= A(u)\cdot u^-\geq\displaystyle{\int_{\mathbb{R}^{2N}}\frac{\vert u^{-}(x)-u^{-}(y)\vert^{p}}{\vert x-y\vert^{N+sp}}}\dd x\dd y=\|u^-\|_{X_p^s}^p
\end{align*}
and therefore $u^-=0$. Thus, a critical point $u$ of $I_{\lambda,s}^+$ satisfies $u=u^+\geq 0$.

As in the proof of Lemma \ref{C3P_3}, we have that $I_{\lambda,s}^+$ satisfy the $(PS)$-condition at any level 
\begin{equation}\label{PS+}
C<\dfrac{s}{N}b^\frac{sp-N}{sp}S_{p,s}^\frac{N}{sp},
\end{equation}
for any $\lambda>0$. 

\begin{lemma}\label{CPL2_24.15}
If $a,b>0$ and $1<q<p$, then the trivial solution $u=0$ is a strict local minimizer of $I_{\lambda,s}$ 
for any $\lambda>0$.
\end{lemma}
\begin{proof}
According to Remark \ref{strict}, it is enough to show that $u=0$ is a strict local minimum of 
$I_{\lambda,s}^+$ in $C^0_{\delta}(\overline{\Omega})$. Take $u\in C^0_\delta(\overline{\Omega})\backslash \{0 \}$
and consider the functional
\begin{align*}
I_{\lambda,s}^+(u)=\frac{1}{p}\|u\|_{X_p^s}^p + \frac{\lambda}{q}\int_{\Omega}\vert u^+\vert^q\dd x-
 \frac{a}{p}\int_{\Omega}\vert u^+\vert^p\dd x-\frac{b}{p^*_s}\int_{\Omega}\vert u^+\vert^{p^*_s}\dd x.
\end{align*}

So, for positive constants  $C_1$ and $C_2$, we have
\begin{align*}
I_{\lambda,s}^+(u)& \geq \frac{1}{p}\|u\|_{X_p^s}^p +
\left( \frac{\lambda}{q}-\frac{aC_1}{p}\| u\|_{0,\delta}^{p-q}-\frac{bC_2}{p^*_s}\| u\|_{0,\delta}^{p^*_s-q}\right)
\int_\Omega\vert u^+\vert^q \dd x
\end{align*}
This  implies for all $u \neq 0$ with $\| u\|_{0,\delta}$ sufficiently small,
$$\hspace*{4.6cm}I_{\lambda,s}^+(u)> 0=I_{\lambda,s}^+(0).\hspace*{4.6cm}\Box$$
\end{proof}

\begin{lemma}\label{CPL2_24.21}
If $\lambda_1<a$, $1<q<p$ and $b>0$, then, for any fixed $\Lambda>0$, there exists $t_0=t_0(\Lambda)>0$ such that
$$I_{\lambda,s}^+(t\varphi_1)<0,$$ for all $t\geq t_0$ and $\lambda<\Lambda$.
\end{lemma}

\begin{proof}
For a fixed $\Lambda>0$, our hypotheses guarantee that we can choose $t_0=t_0(\Lambda)>0$ such that, if $t\geq t_0$ and $\lambda<\Lambda$, it follows $I_{\lambda,s}^+(t\varphi_1)<0$. 
$\hfill\Box$\end{proof}

Now we prove that \eqref{principal} has at least one positive solution. 
\begin{proposition}\label{C3Q_7.0}
Suppose that $\lambda>0$, $1<q<p$, $\lambda_1<a$ and $b>0$. There exists $\lambda_0>0$ such that, if
\begin{equation*}
0<\lambda<\lambda_0,
\end{equation*}
then problem \eqref{principal} has at least one positive solution.
\end{proposition}
\begin{proof}We observe that a non-negative weak solution of \eqref{principal} is a critical point of the functional $I_{\lambda,s}^+$. We now apply the Mountain Pass Theorem. The geometric conditions of this theorem are consequences of Lemmas \ref{CPL2_24.15} and \ref{CPL2_24.21}. We now prove the existence of $\lambda_0>0$ such that, if $0<\lambda<\lambda_0$, then $I_{\lambda,s}^+$ satisfies the $(PS)$-condition at level
$$C_\lambda^+=\displaystyle\inf_{g\in\Gamma^+}\displaystyle\max_{u\in g([0,1])} {I_{\lambda,s}^+(u)},$$ 
where $\Gamma^+=\{g\in C([0,1],{X_p^s})\,:\,g(0)=0,\ g(1)=t_0\varphi_1\}$, with $t_0$ obtained in Lemma \ref{CPL2_24.21}. 

In order to do that, we observe that, for all $0\leq t\leq 1$, our hypotheses imply that 
$$\displaystyle\inf_{g\in\Gamma^+}\displaystyle\max_{u\in g([0,1])}{I_{\lambda,s}^+(u)}\leq\displaystyle\max_{u\in g_0([0,1])} {I_{\lambda,s}^+(u)}=
\displaystyle\max_{t\in [0,1]} {I_{\lambda,s}^+(g_0(t))}\leq \frac{\lambda t_0^q}{q}\| \varphi_1\|_{L^q(\Omega)}^q, $$
from what follows the existence of $\lambda_0>0$ such that 
$$0 \leq C_\lambda^+<\dfrac{s}{N}b^\frac{sp-N}{sp}S_{p,s}^\frac{N}{sp},\ \textrm{ for all }\ 0<\lambda<\lambda_0<\Lambda,$$
with $\Lambda$ as in Lemma \ref{CPL2_24.21}. The proof is complete as a consequence of \eqref{PS+}.
$\hfill\Box$\end{proof}\vspace*{.3cm}

In order to show the existence of a negative solution for $I_{\lambda,s}$, we consider 
$I_{\lambda,s}^-:{X_p^s}\to \mathbb{R}$ given by $$I_{\lambda,s}^-(u)=\frac{1}{p}\|u\|_{X_p^s}^p+\frac{\lambda}{q}\int_{\Omega} \vert u^-\vert^q\dd x-\frac{a}{p}\int_{\Omega}\vert u^-\vert^p\dd x.$$
where $u^-= \min\{ u,0\}$.

Of course $I_{\lambda,s}^-\in C^1({X_p^s},\mathbb{R})$ and critical points of  $I_{\lambda,s}^-$ are weak solutions to the problem 
\begin{equation}\label{CPL2_24.23}
\left\{\begin{array}{rcll}
(-\Delta)^{s}_{p} u&=&-\lambda\vert u^-\vert^{q-1}+ a\vert u^{-}\vert^{p-1}&\textrm{in }\ \Omega,\\
u&=&0 &\textrm{in }\ \mathbb{R}^{N}\setminus \Omega,\end{array}\right.
\end{equation}
As before, a critical point $u$ of $I_{\lambda,s}^-$ satisfies $u=u^-\leq 0$. 

Observe that $I_{\lambda,s}^{-}$ satisfies the $(PS)$-condition at all levels, for any $\lambda>0$, since the nonlinearity in 
\eqref{CPL2_24.23} does not have fractional critical power. 

\begin{lemma}
If $0<a$ and $1<q<p$, then $u=0$ is a strict local minimizer of $I_{\lambda,s}^-$ for all $\lambda>0$.
\end{lemma}
\begin{proof}
According to Remark \ref{strict}, it is enough to show that $u=0$ is a strict local minimum of
$I_{\lambda,s}^-$ in $C^0_{\delta}(\overline{\Omega})$. It is not difficult to verify that, 
for all $u\in C^0_\delta(\overline{\Omega})\setminus\{0\}$ and a fixed $\lambda>0$, we have
\begin{align*}
I_{\lambda,s}^-(u)&\geq \frac{1}{p}\|u\|_{X_p^s}^p +
\left( \frac{\lambda}{q}-\frac{a}{p}\| u\|_{0,\delta}^{p-q}\right) \int_\Omega\vert u^-\vert^q \dd x,
\end{align*}
Taking $R=\Big(\dfrac{\lambda p}{qa}\Big)^{\frac{1}{p-q}}$, it follows 
$a\| u\|_{0,\delta}^{p-q}/p<\lambda/q$ and, for all $u \neq 0$,
$$I_{\lambda,s}^-(u)> 0=I_{\lambda,s}^-(0),\ \  \mbox{if}\ \  \| u\|_{0,\delta}<R. $$
\vspace*{-.5cm}$\hfill\Box$\end{proof}\hspace*{.5cm}

The proof of the next result is analogous to that of Lemma \ref{CPL2_24.21}.
\begin{lemma}
If $\lambda_1<a$ and $1<q<p$, then for any fixed $\Lambda>0$, there exists $t'_0=t'_0(\Lambda)>0$ such that
$$I_{\lambda,s}^-(-t\varphi_1)<0,$$ for all $t\geq t'_0$ and $\lambda<\Lambda$.	
\end{lemma}

The proof of the next result is similar to that of Proposition \ref{C3Q_7.0}. In the proof, the inequality
\begin{equation*}
0\leq C_\lambda^-\leq\frac{\lambda (t'_0)^q}{q}\| \varphi_1\|_{L^q(\Omega)}^q,\quad\textrm{for all }\, \lambda>0,
\end{equation*}
play an essential role, with $C_\lambda^-$ defined analogously to $C_\lambda^+$. 
\begin{proposition}
Suppose that $\lambda>0$, $1<q<p$, $\lambda_1<a$ and $b>0$. There exists $\lambda_0>0$ such that, if
\begin{equation*}
0<\lambda<\lambda_0,
\end{equation*}
then problem \eqref{principal} has at least one negative solution.
\end{proposition}
\section{A third solution via Linking Theorem}
In contrast to the previous section, the proof of the existence of a third solution to \eqref{principal} is much more intricate and also technical.  We obtain a third solution by applying the Linking Theorem and a series of previous obtained results that will be useful in our proof.  

We suppose $0<s<1$, $N>sp$, $\lambda>0$, $\lambda_1<a<\lambda^*$ and $b>0.$

The proof of our first result is simple.
\begin{proposition}\label{decomposition}If $\textrm{span}\{\varphi_1\}$ denotes the space generated by the first (positive, $L^p$-normalized) eigenfunction of $(-\Delta)_p^{s}$, then
$$X_p^s=W\oplus\textup{span}\{\varphi_1\}.$$
\end{proposition}
We recall that  
\begin{equation*}
W=\left\{ u\in X_p^s\; :\; A(\varphi_1)\cdot u=0\right\}\ \ \text{and}\ \ \lambda^{*}=\inf\left\{\|u\|_{X_p^s}^{p}\; : \; u\in W, \ \|u\|_{L^{p}(\Omega)}^{p}=1\right\}.
\end{equation*}

The next result will be used to prove that the geometric conditions of the Linking Theorem are satisfied.
\begin{proposition}\label{C2P_25}
Suppose that $a<\lambda^{*}$. Then, there exist $\alpha>0$ and $\rho>0$ such that $I_{\lambda,s}(u)\geq\alpha$ for any 
$u\in W$ with $\|u\|_{X_p^s}=\rho$.
\end{proposition}
\begin{proof} Since $a>0$, the immersion ${X_p^s}\hookrightarrow L^r(\Omega)$ for $r\in[p,p_s^*]$ combined with the definition of $\lambda^*$ implies (after some calculations) that 
\begin{align*}
I_{\lambda,s}(u)
&\geq\frac{1}{p}\left(1-\frac{a}{\lambda^{*}}\right)\|u\|_{X_p^s}^p-\frac{bC^{p^*_s}}{p^*_s}\| u\|_{X_p^s}^{p^*_s}\geq\|u\|_{X_p^s}^p(A-B\|u\|_{X_p^s}^{p^*_s-p}),
\end{align*}
where $A=\left(1-a/\lambda^{*}\right)/p>0$ and $B=(bC^{p^*_s})/p^*_s>0$. If  $0<\rho<\left(A/B\right)^{1/(p^*_s-p)}$, then 
$$I_{\lambda,s}(u)\geq\rho^p(A-B\rho^{p^*_s-p})=\alpha>0$$
and we conclude that $ I_{\lambda,s}(u)\geq \alpha$ for all $u\in W$ satisfying $\|u\|_{X_p^s}=\rho$.
$\hfill\Box$\end{proof}\vspace*{.5cm}

In order to apply the Linking Theorem with respect to the decomposition given in Proposition \ref{decomposition}, we need to prove the existence of a vector $e\in W$ satisfying the hypotheses of that result. We recall that $S_{p,s}$ was defined in \eqref{C3N_9}.

We now state the following result, which can be found in \cite[Proposition 2.1]{mosconi}. See also \cite{Brasco}.
\begin{proposition}\label{pmosconi}
Let $1 < p <\infty$, $s\in(0, 1)$, $N>sp$. Then
\begin{enumerate}
\item[$(i)$] There exists a minimizer for $S_{p,s}$. 
\item[$(ii)$] For every minimizer $U$, there exist $x_0\in \mathbb{R}^N$ and a constant sign monotone function 
$u :[0,\infty)\to \mathbb{R}$ such that $U (x) = u(|x-x_0 |)$.
\item[$(iii)$] For every nonnegative minimizer $U\in X^s_p$ and $v\in {X_p^s}$, we have
$$\int_{\mathbb{R}^{2N}}\frac{|U(x)-U(y)|^{p-2}\left(U(x)-U(y)\right)(v(x)-v(y))}{\vert x-y\vert^{N+sp}}\dd x\dd y=S_{p,s}\int_{\mathbb{R}^{2N}}U^{p^*_s-1}v\dd x.$$
\end{enumerate}
\end{proposition}\vspace*{.5cm}

Applying Proposition \ref{pmosconi}, we fix a radially symmetric nonnegative decreasing minimizer $U=U(r)$ of $S_{p,s}$. 
Multiplying $U$ by a positive constant, we may assume that
\begin{equation}\label{C3N_10}
(-\Delta)_p^sU=U^{p^*_s-1}.
\end{equation}
It follows from \eqref{C3N_9} that 
\begin{equation}\label{C3N_11}
\|U\|_{X_p^s}^p=\|U\|_{L^{p_s^*}(\mathbb{R}^N)}^{p_s^*}=\left(S_{p,s}\right)^{N/sp}.
\end{equation}
For any $\varepsilon>0$, the function
\begin{equation}\label{C3N_12}
U_{\varepsilon}(x)=\frac{1}{\varepsilon^{(N-sp)/p}}U\left(\frac{\vert x\vert}{\varepsilon}\right)
\end{equation}
is also a minimizer of $S_{p,s}$ satisfying  \eqref{C3N_10} and \eqref{C3N_11}, so after a rescaling we may assume that $U(0)=1$. Henceforth, $U$ will denote such a function and $\{U_\varepsilon\}_{\varepsilon>0}$ the associated family of minimizers given by \eqref{C3N_12}.

Since an explicit formula for a minimizer of $S_{p,s}$ is unknown, we make use of some asymptotic estimates obtained by Brasco, Mosconi and Squassina \cite{Brasco}, see also \cite[Lemma 2.2]{mosconi}.
\begin{lemma}\label{C3N_13}
There exist constants $c_1, c_2>0$ and $\theta>1$ such that for all $r\geq 1$,
\begin{equation*}
\frac{c_1}{r^{(N-sp)/(p-1)}}\leq U(r)\leq\frac{c_2}{r^{(N-sp)/(p-1)}} \quad \textrm{and}\quad  \frac{U(\theta r)}{U(r)}\leq\frac{1}{2}.
\end{equation*}
\end{lemma}\vspace*{.5cm}

In order to apply the Linking Theorem with respect to the decomposition given by Proposition \ref{decomposition}, we consider the family of functions $\{U_\varepsilon\}_{\varepsilon>0}$ as defined in \eqref{C3N_12}. Without loss of generality, we suppose that $0 \in\Omega$. From now on, let us consider $\theta$ as given in Lemma \ref{C3N_13}. For $\varepsilon, \delta>0$ we define, as in Chen, Mosconi and Squassina \cite{Chen},
$$m_{\varepsilon,\delta}=\frac{U_{\varepsilon}(\delta)}{U_{\varepsilon}(\delta)-U_{\varepsilon}(\theta\delta)},$$ and also
$$g_{\varepsilon,\delta}(t):= \left\{ \begin{array}{ll}
0, &\textrm{if}\quad 0\leq t\leq U_\varepsilon(\theta\delta),\\
m^p_{\varepsilon,\delta}(t- U_\varepsilon(\theta\delta)), &\textrm{if}\quad U_\varepsilon(\theta\delta)\leq t\leq U_{\varepsilon}(\delta),\\
t-U_{\varepsilon}(\delta)(m_{\varepsilon,\delta}^{p-1}-1),&\textrm{if}\quad t\geq U_{\varepsilon}(\delta),
\end{array} \right.$$ 

Since
\begin{align*}
G_{\varepsilon,\delta}(t)&:=\int_{0}^t(g'_{\varepsilon,\delta}(\tau))^{\frac{1}{p}}d\tau=\left\{ \begin{array}{ll}
0, &\textrm{if}\quad 0\leq t\leq U_\varepsilon(\theta\delta),\\
m_{\varepsilon,\delta}(t- U_\varepsilon(\theta\delta)), &\textrm{if}\quad U_\varepsilon(\theta\delta)\leq t\leq U_{\varepsilon}(\delta),\\
t,&\textrm{if}\quad t\geq U_{\varepsilon}(\delta),
\end{array} \right.
\end{align*}
it is not difficult to verify that the functions $g_{\varepsilon,\delta}$ and $G_{\varepsilon,\delta}$ are non-decreasing and absolutely continuous.  

We now define the non-increasing, absolutely continuous and radially symmetric function
$$u_{\varepsilon,\delta}(r)=G_{\varepsilon,\delta}(U_{\varepsilon}(r))$$
which satisfies
\begin{equation}\label{C3N_19}
u_{\varepsilon,\delta}(r)= \left\{ \begin{array}{ll}
U_{\varepsilon}(r), &\textrm{if}\quad r\leq \delta,\\
0, &\textrm{if}\quad r\geq \theta\delta.
\end{array} \right.
\end{equation}
According to Ambrosio and Isernia \cite[p. 17]{AM}, for any $\delta\leq r\leq\theta\delta$ it holds
$$0\leq m_{\varepsilon,\delta}\big(U_{\varepsilon}(r)-U_{\varepsilon}(\theta\delta)\big)=U_{\varepsilon}(\delta)\left[\dfrac{U_{\varepsilon}(r)-U_{\varepsilon}(\theta\delta)}{U_{\varepsilon}(\delta)-U_{\varepsilon}(\theta\delta)}\right]\leq U_{\varepsilon}(\delta).$$
Therefore, from the definition of $G_{\varepsilon,\delta}$ and \eqref{C3N_19} follows that
\begin{equation*}
u_{\varepsilon,\delta}(r)\leq\left\{ \begin{array}{ll}
U_{\varepsilon}(r), &\textrm{if}\quad r< \theta\delta,\\
0, &\textrm{if}\quad r\geq \theta\delta.
\end{array} \right.
\end{equation*}

We denote by $P_{1}^{s}$ and $P_{2}^{s}$ the projections of ${X_p^s}$  in $\textrm{span}\{\varphi_1\}$ and $W$, respectively, and define
\begin{equation*}
e_{\varepsilon,\delta}=P_{2}^{s}u_{\varepsilon,\delta}\in W, 
\end{equation*}
and claim that $e_{\varepsilon,\delta}$ is a continuous function. 
As shown in \cite{Brasco}, we know that $U\in L^{\infty}(\mathbb{R}^N)\cap C^{0}(\mathbb{R}^N)$.
Since $e_{\varepsilon,\delta}=u_{\varepsilon,\delta}-P_{1}^{s}u_{\varepsilon,\delta}$, our claim is proved.

We now want to show that we can take $e=e_{\varepsilon,\delta}$ in the Linking Theorem. 
So, we need to show that $e_{\varepsilon,\delta}(0)>0$ for $\varepsilon>0$ sufficiently small. In order to do that, 
we obtain the inequalities of the next result using arguments similar to those used in \cite{Ch}, with the exception of 
\eqref {C3N_23}. The  estimate \eqref {C3N_23} follows from \cite[Lemma 2.4]{AM}. Therefore, we state the following result.
\begin{lemma}
\begin{equation}\label{C3N_22}
\| P_{1}^{s} u_{\varepsilon,\delta}\|_{L^{\infty}(\Omega)}\leq\left\{ \begin{array}{ll}
\| \varphi_1\|_{L^\infty(\Omega)}^pC_1 \varepsilon^{\frac{N}{p}}\vert\log(\frac{\varepsilon}{\delta})\vert, &\textrm{if}\quad  p=\frac{2N}{N+s}\\
\|\varphi_1\|_{L^\infty(\Omega)}^pC_1\varepsilon^{N-\frac{N-sp}{p}}, &\textrm{if}\quad 1< p<\frac{2N}{N+s},\\
\|\varphi_1\|_{L^\infty(\Omega)}^pC_1\delta^{N-\frac{N-sp}{p-1}}\varepsilon^{\frac{N-sp}{p(p-1)}}, &\textrm{if}\quad p> \frac{2N}{N+s}.
\end{array} \right.
\end{equation}
\begin{equation}\label{C3N_22.0}
\|e_{\varepsilon,\delta}\|_{L^{1}(\Omega)}\leq\left\{ \begin{array}{ll}
C_0\varepsilon^{\frac{N}{p}}\vert\log(\frac{\varepsilon}{\delta})\vert, &\textrm{if}\quad  p=\frac{2N}{N+s}\\
C_0\varepsilon^{N-\frac{N-sp}{p}}, &\textrm{if}\quad 1< p<\frac{2N}{N+s},\\
C_0\delta^{N-\frac{N-sp}{p-1}}\varepsilon^{\frac{N-sp}{p(p-1)}}, &\textrm{if}\quad p> \frac{2N}{N+s},
\end{array} \right.
\end{equation}
\begin{equation}\label{C3N_23}
\|e_{\varepsilon,\delta}\|_{L^{p}(\Omega)}^p\leq\left\{ \begin{array}{ll}
K_1\varepsilon^{sp}\vert\log(\frac{\varepsilon}{\delta})\vert^p, &\textrm{if}\quad  p=\frac{2N}{N+s}\,\textrm{ and }\,N>sp^2\\
K_1\varepsilon^{sp}, &\textrm{if}\quad 1< p<\frac{2N}{N+s}\,\textrm{ and }\,N>sp^2\\
K_1\varepsilon^{sp}, &\textrm{if}\quad p> \frac{2N}{N+s}\,\textrm{ and }\,N>sp^2,
\end{array} \right.
\end{equation}
\begin{equation}\label{C3N_24}
\|e_{\varepsilon,\delta}\|_{L^{p^*_s-1}(\Omega)}^{p^*_s-1}\leq\left\{ \begin{array}{ll}
K_2\varepsilon^{\frac{N-sp}{p}}\vert\log(\frac{\varepsilon}{\delta})\vert^{p^*_s-1}, &\textrm{if}\quad  p=\frac{2N}{N+s},\\
K_2\varepsilon^{\frac{N-sp}{p}}, &\textrm{if}\quad 1< p<\frac{2N}{N+s},\\
K_2\varepsilon^{\frac{N-sp}{p} }, &\textrm{if}\quad p> \frac{2N}{N+s},
\end{array} \right.
\end{equation}
\begin{equation*}
\hspace*{-.4cm}\left\vert\int_{\Omega}\left(\vert e_{\varepsilon,\delta}\vert^{p^*_s}-\vert u_{\varepsilon,\delta}\vert^{p^*_s}\right)\dd x\right|\leq \left\{\begin{array}{ll}
K_3\varepsilon^{N}\vert\log(\frac{\varepsilon}{\delta})\vert^{p^*_s}, &\textrm{if}\quad  p=\frac{2N}{N+s},\\
K_3\varepsilon^{N}+K_3\varepsilon^{N(p_s^*-1)}, &\textrm{if}\quad 1< p<\frac{2N}{N+s},\\
K_3\varepsilon^{\frac{N-sp}{p-1}}, &\textrm{if}\quad p> \frac{2N}{N+s}.
\end{array} \right.
\end{equation*}
\end{lemma}

We now fix $K>0$. 
\begin{lemma}\label{lema1}
There exist $\varepsilon(K)>0$ and $\sigma>0$ such that
\begin{equation*}
B_{\sigma}(0)\subset\left\{ x\in\Omega\, :\, e_{\varepsilon,\delta}(x)>K\right\}:=\Omega_{\varepsilon,K}
\end{equation*}
for all $0<\varepsilon\leq\varepsilon(K)$.	

As a consequence, 
\begin{equation*}
\left\vert\displaystyle\int_{\Omega_{\varepsilon,K}}\vert e_{\varepsilon,\delta}\vert^{p^*_s}\dd x-\int_{\Omega}\vert u_{\varepsilon,\delta}\vert^{p^*_s} \dd x \right\vert\leq \left\{ \begin{array}{ll}
K_4\varepsilon^{N}\vert\log(\frac{\varepsilon}{\delta})\vert^{p^*_s}, &\textrm{if}\quad  p=\frac{2N}{N+s},\\
K_4\varepsilon^{N}+K_4\varepsilon^{N(p_s^*-1)}, &\textrm{if}\quad 1< p<\frac{2N}{N+s},\\
K_4\varepsilon^{\frac{N-sp}{p-1}}, &\textrm{if}\quad p> \frac{2N}{N+s},
\end{array} \right.
\end{equation*}
\end{lemma}
\begin{proof}It follows from the definition of  $u_{\varepsilon,\delta}$ , Lemma \ref{C3N_13}  and \eqref{C3N_22}  that
\begin{equation}\label{C3N_26.0}
e_{\varepsilon,\delta}(0)\geq\dfrac{1}{\varepsilon^{(N-sp)/p}}-\| P_{1}^{s} u_{\varepsilon,\delta}\|_{L^{\infty}(\Omega)}\to+\infty,
\end{equation}
as $\varepsilon\to0$, the proof of the claim is complete, because $e_{\varepsilon,\delta}$ is continuous.
The proof of the estimates are obtained by applying Lemma 2.4 in de Figueiredo and Yang \cite{deFiguei}.
$\hfill\Box$\end{proof}\vspace*{.5cm}

It follows from Lemma \ref{lema1} that there exists $\varepsilon_0>0$ such that
\begin{equation*}
e_{\varepsilon,\delta}\neq 0,\ 
0<\varepsilon\leq\varepsilon_0.
\end{equation*}

Thus, in the Linking Theorem, we can take  $e=e_{\varepsilon,\delta}$. 

If $\varepsilon\in (0,\varepsilon_0]$, take $R_1, R_2>0$ and define
\begin{equation}\label{C3N_28}
\hspace*{-.2cm}Q_{\varepsilon,R_1,R_2}=\{ u\in {X_p^s}: u=u_1+re_{\varepsilon,\delta},\, u_1\in \mbox{span}\{\varphi_1\}\cap\overline{B}_{R_1}(0),\, 0\leq r\leq R_2\}. 
\end{equation}

Let $\partial Q_{\varepsilon,R_1,R_2}$ be the boundary of $Q_{\varepsilon,R_1,R_2}$ in the finite dimensional space $\textrm{span}\{\varphi_1\}\,\oplus \textrm{span}\{e_{\varepsilon,\delta}\}$. 

We denote $O(\varepsilon^{\omega})$ for $\omega\geq 0$ if $\vert O(\varepsilon^{\omega})\vert\leq C\varepsilon^{\omega}$ for some $C>0
$ not depending on $\varepsilon>0$. We remark that $O(\varepsilon^{\omega})$ is not always positive.

The next elementary result can be found in Mosconi, Perera, Squassina and Yang \cite[p. 17]{mosconi}.
\begin{lemma}\label{C3N_34}
Given $k>1$ and $p-1< \tau < p$, there exists a constant $C = C(k, q) > 0$ such that
\begin{equation*}
\vert a+b\vert^p\leq k\vert a\vert^p+\vert b\vert^p+C\vert a\vert^{p-\tau}\vert b\vert^{\tau},\;\forall\, a,b\in\mathbb{R}.
\end{equation*}
\end{lemma}

As an immediately consequence of Lemma \ref{C3N_34}, we have
\begin{lemma}\label{C3N_35}
There exist constants $C_1, C_2>0$ such that 
\begin{equation*}
\vert a+b\vert^p\leq C_1\vert a\vert^p+C_2\vert b\vert^p,\;\forall\, a,b\in\mathbb{R}.
\end{equation*}
\end{lemma}

The proof of the next result is obtained by applying Lemma \ref{C3N_34}	with $a=\| r e_{\varepsilon,\delta}\|_{L^p(\Omega)}$ and $b=\|u_1+re_{\varepsilon,\delta}\|_{L^p(\Omega)}$ and considering the cases $0<\tau<1$ and $\tau>1$. In the latter case we then apply Lemma \ref{C3N_35} and once again Lemma \ref{C3N_34} with $a=\|r e_{\varepsilon,\delta}\|$ and $b=\|u_1\|_{X_p^s}$.
\begin{lemma}\label{C3N_42}
Suppose that $u_1\in \textrm{span}\{\varphi_1\}$ and $\tau \in (p-1,p)$. Then there exists a constant $C_{*}>0$ such that
\begin{align*}
\dfrac{1}{p}\|u_1+re_{\varepsilon,\delta}\|_{X_p^s}^p-\dfrac{a}{p}\|u_1+r e_{\varepsilon,\delta}\|_{L^p(\Omega)}^p &\leq \dfrac{1}{p}\|u_1\|_{X_p^s}^p-\dfrac{a}{p}\|u_1\|_{L^p(\Omega)}^p+C_{*}r^p\|e_{\varepsilon,\delta}\|_{L^p(\Omega)}^p\\
&\quad+C_{*}r^{p-\tau}\|e_{\varepsilon,\delta}\|_{L^p(\Omega)}^{p-\tau}\|u_1\|_{X_p^s}^\tau.
\end{align*}
\end{lemma}

The next result follows immediately by considering $f:[0,\infty)\to\mathbb{R}$ given by $f(t)=Bt^p/p-Ct^{p_s^*}/p_s^*$.
\begin{lemma}\label{C5P_33.0.0}
For constants $B>0$ and $C>0$ we have $$\displaystyle\max_{t\geq 0}\left(\frac{Bt^p}{p}-\frac{Ct^{p_s^*}}{p_s^*}\right)=
\frac{s}{N}\left(\frac{B}{C^\frac{p}{p_s^*}}\right)^{\frac{N}{sp}}.$$
\end{lemma}

To estimate $I_{\lambda,s}$ on  $\partial Q_{\varepsilon,R_1,R_2}$ we apply Lemmas \ref{C3N_34} to \ref{C5P_33.0.0} and the two next results. The first is a special case of the one proved by Chen, Mosconi and Squassina \cite[Lemma 2.11]{Chen}, while the proof of the second can be found in \cite[Lemma 2.7]{mosconi}.
\begin{lemma}\label{a.5}For any $\beta>0$ and $0<2\varepsilon\leq\delta<\theta^{-1}\textup{dist}(0,\partial\Omega)$ we have
\begin{equation*}
\|u_{\varepsilon,\delta}\|_{L^\beta(\Omega)}\leq\left\{ \begin{array}{ll}
C_\beta \varepsilon^{N-\frac{N-ps}{p}\beta}\vert\log(\frac{\varepsilon}{\delta})\vert, &\textrm{if}\quad 
\beta=\frac{p^*_s}{p'},\\
C_\beta\varepsilon^{\frac{N-sp}{p(p-1)}\beta}\delta^{N-\frac{N-sp}{p-1}\beta}, &\textrm{if}\quad \beta<\frac{p^*_s}{p'},\\
C_\beta\delta^{N-\frac{N-sp}{p}\beta}, &\textrm{if}\quad \beta>\frac{p^*_s}{p'},
\end{array} \right.
\end{equation*}
where $p'=p/(p-1)$.
\end{lemma}
Observe that, taking $\beta=1$, 
we obtain \begin{equation}\label{C3N_20}
\|u_{\varepsilon,\delta}\|_{L^1(\Omega)}\leq\left\{ \begin{array}{ll}
C_1 \varepsilon^{\frac{N}{p}}\vert\log(\frac{\varepsilon}{\delta})\vert, &\textrm{if}\quad  p=\frac{2N}{N+s},\\
C_1\varepsilon^{N-\frac{N-sp}{p}}, &\textrm{if}\quad 1< p<\frac{2N}{N+s},\\
C_1\delta^{N-\frac{N-sp}{p-1}}\varepsilon^{\frac{N-sp}{p(p-1)}}, &\textrm{if}\quad p> \frac{2N}{N+s}.
\end{array} \right.
\end{equation}

Since $(p^*_s-1)p/(p-1)>p_s^*$, it also follows from Lemma \ref{a.5} by taking $\beta=p^*_s-1$ that
\begin{equation}\label{C3N_21}
\|u_{\varepsilon,\delta}\|_{L^{p^*_s-1}(\Omega)}^{p^*_s-1}\leq C_{p^*_{s}-1} \varepsilon^{\frac{N-sp}{p}}.
\end{equation}

\begin{lemma} \label{a.6}
There exists a constant $C = C(N, p, s) > 0$ such that for any $0<\varepsilon\leq \delta/2$,
\begin{itemize}
\item[$(i)$]$\|u_{\varepsilon, \delta}\|_{W_0^{s,p}}^{p} \leq S_{p,s}^{N / s p}+C\left(\frac{\varepsilon}{\delta}\right)^{(N-s p) /(p-1)}.$
\item[$(ii)$] ${\|u_{\varepsilon, \delta}\|_{ L^p(\Omega)}^{p} \geq\left\{\begin{array}{ll}{\frac{1}{C} \varepsilon^{s p}\log \left(\frac{\delta}{\varepsilon}\right),} 
& {N=s p^{2}},\vspace*{.2cm} \\ {\frac{1}{C} \varepsilon^{s p},} & {N>s p^{2}}.
\end{array}\right.}$
\item[$(iii)$] $\|u_{\varepsilon, \delta}\|_{L^{p_{s}^{*}}(\Omega)}^{p_{s}^{*}} \geq S_{p,s}^{N / s p}-C\left(\frac{\varepsilon}{\delta}\right)^{N /(p-1)}.$
\end{itemize}
\end{lemma}

In our development, we need an auxiliary result that was obtained in \cite[Lemma 2.5]{deFiguei} when $s = 1$ and $p=2$. It is easily adapted for 
$s\in (0, 1)$ and $p>1$.
\begin{lemma}\label{C5P_33.0}
Let $u,v\in L^r(\Omega)$ with $p\leq r\leq p_s^*.$ If $\omega\subset\Omega$ and $u+v>0$ on $\omega,$ then
$$\bigg\vert\int_{\omega}(u+v)^r \dd x-\int_{\omega}\vert u\vert^r \dd x-\int_{\omega}\vert v\vert^r \dd x\bigg\vert\leq 
C\int_{\omega}\big(\vert u\vert^{r-1}\vert v\vert+\vert u\vert\vert v\vert^{r-1}\big) \dd x,$$ where $C$ depends only on $r.$
\end{lemma}

Adapting ideas from Miyagaki, Motreanu and Pereira \cite{O} and de Figueiredo and Yang \cite{deFiguei}, we obtain the desired  estimate $I_{\lambda,s}$ on $\partial Q_{\varepsilon,R_1,R_2}$.
\begin{proposition}\label{C3N_28.0} Consider $Q_{\varepsilon,R_1,R_2}$ defined in \eqref{C3N_28}. There exist $R_1>0$ and $R_2>0$ large enough such that 
\begin{equation*}
I_{\lambda,s}(u)\leq\frac{\lambda}{q}\int_{\Omega}\vert u\vert^q \dd x,\,\forall\, u\in\partial Q_{\varepsilon,R_1,R_2},
\end{equation*}
for all $\varepsilon>0$ small enough and all $\lambda>0$.
\end{proposition}
\begin{proof} We write $\partial Q_{\varepsilon,R_1,R_2}=\Gamma_1\cup\Gamma_2\cup\Gamma_3$,
with
\begin{align*}
\Gamma_1&=B_{R_1}\cap\textrm{span}\{\varphi_1\},\\
\Gamma_2&=\{ u\in {X_p^s}\, : \, u=u_1+r e_{\varepsilon,\delta},\; u_1\in \textrm{span}\{\varphi_1\},\;\|u_1\|_{X_p^s}=R_1,\; 0\leq r\leq R_2\},\\
\Gamma_3&=\{ u\in {X_p^s}\, : \, u=u_1+R_2 e_{\varepsilon,\delta},\; u_1\in \textrm{span}\{\varphi_1\},\;\|u_1\|_{{X_p^s}}\leq R_1\}.
\end{align*}

We consider $I_{\lambda,s}$ in the three parts of the boundary $Q_{\delta,R_1,R_2}$. If $u\in\Gamma_1$, then $u=t\varphi_1$ and we obtain
\begin{align*}
I_{\lambda,s}(u)&\leq\frac{\vert t\vert^p}{p}(\lambda_1-a)\int_{\Omega}\vert \varphi_1\vert^p\dd x
+\frac{\lambda}{q}\int_{\Omega}\vert t \varphi_1\vert^{q}\dd x\leq\frac{\lambda}{q}\int_{\Omega}\vert u\vert^q \dd x,\,\forall\, u\in\Gamma_1,
\end{align*}
as desired.

If $u\in\Gamma_2$, then $u=u_1+r e_{\varepsilon,\delta}\in \textrm{span}\{\varphi_1\}\,\oplus \textrm{span}\{e_{\varepsilon,\delta}\}$, with  $\|u_1\|_{{X_p^s}}=R_1$. Since $P_{2}^{s}$ is bounded in ${X_p^s}$, there exist $C_1>0$ such that $$\|e_{\varepsilon,\delta}\|_{{X_p^s}}=\|P_{2}^{s}u_{\varepsilon,\delta}\|_{{X_p^s}}\leq C_1\|u_{\varepsilon,\delta}\|_{{X_p^s}}.$$

It follows then from Lemma \ref{a.6}$(i)$ the existence of $C>0$ such that, for all $0<\varepsilon\leq \widehat{\varepsilon_0}=\min\big\{\delta/2,\varepsilon_0\big\}$, we have
\begin{equation}\label{eq002}
\|e_{\varepsilon,\delta}\|_{{X_p^s}}^p\leq C_1^p\|u_{\varepsilon,\delta}\|_{{X_p^s}}^p\leq C_1^pS_{p,s}^{N/sp}+C_1^pC\left(\frac{\varepsilon}{\delta}\right)^{(N-sp)/(p-1)}.
\end{equation} 

We conclude that $\eta:=\displaystyle{\sup_{0<\varepsilon\leq\widehat{\varepsilon_0}}}\|e_{\varepsilon,\delta}\|_{{X_p^s}}$ is finite. 

In order to satisfy the condition $R_2\|e_{\varepsilon,\delta}\|>\rho$ in the Linking Theorem for $0<\varepsilon\leq \widehat{\varepsilon_0}$ small enough, with $\delta<\theta^{-1}\textup{dist}(0,\partial\Omega)/2$ as in Lemma \ref{a.5} and $\rho>0$ as in Proposition \ref{C2P_25}, we must have
$$R_2\eta\geq R_2\|e_{\varepsilon,\delta}\|_{{X_p^s}}>\rho,$$
showing that $\rho/\eta$ is a lower bound for $R_2$. 

Thus, we define $r_0=\max\{\rho/\eta, 1\}$ and consider two cases: 

\noindent \textbf{a)}  $0\leq r\leq r_0$. 
Since all norms are equivalent on finite dimensional spaces 
\begin{align*}
I_{\lambda,s}(u)
&\leq \;\frac{1}{p}\left(1-\frac{a}{\lambda_1} \right)R_1^p+C_{*}r_0^p\eta^p+C_{*}r_0^{p-\tau}\eta^{p-\tau}R_1^\tau+\frac{\lambda}{q}\int_{\Omega}\vert u\vert^{q}\dd x.
\end{align*}
Since $a>\lambda_1$, the result is obtained in this case for all $R_1> 0$ large enough to be fixed later.

\noindent\textbf{b)} $r> r_0$. We suppose $R_1\geq 1$ and denote
\begin{equation*}
K(R_1):=\frac{1}{r_0}\sup\left\{ \|u_1\|_{L^{\infty}(\Omega)}\, : \, u_1\in \textrm{span}\{\varphi_1\},\, \|u_1\|_{{X_p^s}}=R_1 \right\}\in[c_0R_1, c_1R_1],
\end{equation*}
with positive constants $c_0$ and $c_1$. We introduce the open set $$\Omega_{\varepsilon,\delta}=\left\{x\in\Omega\, :\, e_{\varepsilon,\delta}(x)>K(R_1)\right\}.$$
It follows from \eqref{C3N_26.0} that $0\in\Omega_{\varepsilon,\delta}$ if  $\varepsilon\in(0,\widehat{\varepsilon_0}]$. 
	
If $x\in\Omega_{\varepsilon,\delta}$ and $r>r_0$, we have $e_{\varepsilon,\delta}(x)>\vert u_1(x)/r\vert\geq-\dfrac{u_1(x)}{r}$
a.e. in  $\Omega$. That is, 
\begin{equation}\label{C3N_47}
e_{\varepsilon,\delta}(x)+\frac{u_1(x)}{r}>0
\end{equation}
for all $x\in\Omega_{\varepsilon,\delta}$, $r>r_0$ and $u_1\in\textrm{span}\{\varphi_1\}$ with $\|u_1\|_{{X_p^s}}=R_1.$

We now take $R_2=2R_1$. By Lemma \ref{C3N_42} and using the estimate \eqref{eq002}, we can write
\begin{align*}
I_{\lambda,s}(u)\leq &\frac{1}{p}\left(1-\frac{a}{\lambda_1} \right)R_1^p+C_{*}r^p\big(C_1^pS_{p,s}^{N/sp}+
C_0\big)+2^{p-\tau}C_{*}R_1^p\|e_{\varepsilon,\delta}\|_{L^p(\Omega)}^{p-\tau} \\
&-\frac{b}{p^*_s}r^{p^*_s}\int_{\Omega_{\varepsilon,\delta}}\left(\frac{u_1}{r}+e_{\varepsilon,\delta}\right)^{p_s^*}+\frac{\lambda}{q}\int_{\Omega}\vert u\vert^{q}\dd x.
\end{align*}
Now, applying Lemma \ref{C5P_33.0} for $u=u_1/r$, $v=e_{\varepsilon,\delta}$ and $\omega=\Omega_{\varepsilon,\delta}$, we obtain
\begin{align*}
I_{\lambda,s}(u)
	\leq &\frac{1}{p}\left(1-\frac{a}{\lambda_1} \right)R_1^p+C_{*}r^p\big(C_1^pS_{p,s}^{N/sp}+C_0\big)+2^{p-\tau}C_{*}R_1^p\|e_{\varepsilon,\delta}\|_{L^p(\Omega)}^{p-\tau} \\&-\frac{b}{p^*_s}r^{p^*_s}\left[\| e_{\varepsilon,\delta}\|_{L^{p^*_s}(\Omega_{\varepsilon,\delta})}^{p^*_s}-C_2\left(\| e_{\varepsilon,\delta}\|_{L^{1}(\Omega)}R_1^{p^*_s-1}+\| e_{\varepsilon,\delta}\|^{p^*_s-1}_{L^{p^*_s-1}(\Omega)}R_1  \right) \right]\\& +\frac{\lambda}{q}\int_{\Omega}\vert u\vert^{q}\dd x
	\end{align*}
for a positive constant $C_2>0.$
	
We consider the case where $p>2N/(N+s)$ and $N>sp((p-1)^2 +p)>sp^2$, since the case
$1<p\leq 2N/(N+s)$ and $N>sp^2$ is analogous.

The estimates \eqref{C3N_22.0},  \eqref{C3N_23}, \eqref{C3N_24} and Lemma \ref{lema1}  combined with Lemma \ref{a.6} imply that
\begin{align*}
I_{\lambda,s}(u)\leq & \;\frac{1}{p}\left[\left(1-\frac{a}{\lambda_1} \right)+2^{p-\tau}C_{*}\varepsilon^{s(p-\tau)} \right]R_1^p+C_{*}r^p\big(C_1^pS_{p,s}^{N/sp}+C_0\big)\\
&-\frac{b}{p^*_s}r^{p^*_s}
\left[S_{p,s}^{N/sp}+O(\varepsilon^{\frac{N-sp}{p-1}})-C_2\varepsilon^{\frac{N-sp}{p(p-1)}-\gamma(p^*_s-1)}-C_2\varepsilon^{\frac{N-sp}{p}-\gamma} \right]\\& +\frac{\lambda}{q}\int_{\Omega}\vert u\vert^{q}\dd x,
\end{align*}
where $R_1=\varepsilon^{-\gamma}$ with $ \gamma$ chosen so that   $0<\gamma<\min\left\{\dfrac{N-sp}{p(p-1)(p_s^*-1)}, \dfrac{N-sp}{p}\right\}.$

Taking 
\begin{align*}
A&=\left(1-\dfrac{a}{\lambda_1} \right)+2^{p-\tau}C_{*}\varepsilon^{s(p-\tau)},\   B=C_{*}p\big(C_1^pS_{p,s}^{N/sp}+C_0\big)\\
\intertext{and}
C&=bS_{p,s}^{N/sp}+O(\varepsilon^{\frac{N-sp}{p-1}})-bC_2\varepsilon^{\frac{N-sp}{p(p-1)}-\gamma(p^*_s-1)}-
bC_2\varepsilon^{\frac{N-sp}{p}-\gamma
},\end{align*} 
it is easy to see that there exists $ \varepsilon_1> 0 $ small enough such that for all $0<\varepsilon<\varepsilon_1<\widehat{\varepsilon_0}$, 
we have $A<0$ and $C>0$, so
\begin{align*}
I_{\lambda,s}(u)\leq & \;\frac{A}{p} R_1^p+\dfrac{B}{p}r^p-\frac{C}{p^*_s}r^{p^*_s} +\frac{\lambda}{q}\int_{\Omega}\vert u\vert^{q}\dd x.
\end{align*}

By applying Lemma \ref{C5P_33.0.0} to the function $h(r)=(Br^p/p)-(Cr^{p^*_s}/p^*_s)$ we obtain
\begin{align*}
I_{\lambda,s}(u)\leq &\frac{A}{p} R_1^p+\frac{\lambda}{q}
\int_{\Omega}\vert u\vert^{q}\dd x\\ &+\dfrac{1}{N}\left(\frac{C_{*}p\big(C_1^pS_{p,s}^{N/sp}+C_0\big)}{ [bS_{p,s}^{N/sp}+O(\varepsilon^{\frac{N-sp}{p-1}})-bC_2\varepsilon^{\frac{N-sp}{p(p-1)}-\gamma(p^*_s-1)}-bC_2\varepsilon^{\frac{N-sp}{p}-\gamma}]^{p/p^*_s}}\right)^{N/sp}.
\end{align*}
Therefore, since $A<0$ and there exists $\widehat{\varepsilon_1}>0$ small enough that if
$0<\varepsilon<\widehat{\varepsilon_1}<\varepsilon_1$ then $ R_1> 0 $ is large enough, we have the result for $u\in\Gamma_2$.

Finally, if $u\in\Gamma_3$, then  $u=u_1+R_2e_{\varepsilon,\delta}\in \textrm{span}\{\varphi_1\}\,\oplus \textrm{span}\{e_{\varepsilon,\delta}\}$, with $\|u_1\|_{{X_p^s}}\leq R_1$. We recall that	$R_2=2R_1$. 

By \eqref{eq002}, since $\|e_{\varepsilon,\delta}\|_{{X_p^s}}$ is bounded, we have
\begin{align}\label{C3N_47.0}
I_{\lambda,s}(u)\leq & \;\left[ C_{*}\big(C_1^pS_{p,s}^{N/sp}+C_0\big)+\dfrac{C_{*}C_2 }{2^{\tau}}\right] R_2^{p}-\frac{b}{p^*_s}R_2^{p^*_s}\int_{\Omega}\left(\left(\frac{u_1}{R_2}+e_{\varepsilon,\delta}\right)^{+}\right)^{p_s^*}\nonumber\\&+\frac{\lambda}{q}\int_{\Omega}\vert u\vert^{q}\dd x.
\end{align}
Since the norms are equivalent on finite dimensional spaces, there exists a constant $C_3>0$ such that, if $\|u_1\|_{{X_p^s}}\leq R_1$, then
\begin{equation*}
\|u_1\|_{L^{\infty}(\Omega)}\leq C_3\|u_1\|_{{X_p^s}}\leq C_3R_1.
\end{equation*}
Furthermore,
	\begin{equation}\label{C3N_48}
	\Omega'_{\varepsilon}:=\left\{x\in\Omega\, :\, e_{\varepsilon,\delta}(x)>\dfrac{C_3}{2}+1  \right\}\supset\left\{x\in\Omega\, :\, e_{\varepsilon,\delta}(x)>C_3+1  \right\}:=D_{\varepsilon}
	\end{equation}
with $\vert D_{\varepsilon}\vert>0$.

So, for all $x\in\Omega'_{\varepsilon}$, it follows from \eqref{C3N_47}, \eqref{C3N_48}  and $ R_2=2R_1$ that
\begin{align}\label{C3N_49}
\dfrac{u_1(x)}{R_2}+e_{\varepsilon,\delta}(x)>\dfrac{u_1(x)}{R_2}+\dfrac{C_3R_1}{R_2}+1\geq\dfrac{u_1(x)}{R_2}+\dfrac{\|u_1\|_{L^{\infty}(\Omega)}}{R_2}+1\geq 1.
\end{align}

Substituting \eqref{C3N_48} and \eqref{C3N_49} into \eqref{C3N_47.0} we obtain
\begin{align*}
I_{\lambda,s}(u)&\leq\left[ C_{*}\big(C_1^pS_{p,s}^{N/sp}+C_0\big)+\dfrac{C_{*}C_2 }{2^{\tau}}\right] R_2^{p}-\frac{b}{p^*_s}R_2^{p^*_s}\vert D_{\varepsilon}\vert+\frac{\lambda}{q}\int_{\Omega}\vert u\vert^{q}\dd x,
\end{align*}
for all $u\in\Gamma_3$. So by choosing $\varepsilon>0$ small enough, we obtain $R_2$ (and also $R_1$) large enough and our proof is complete.
$\hfill\Box$\end{proof}
\begin{remark}
It is noteworthy to stress that $R_1$, $R_2$ and $\varepsilon$ in Proposition \ref{C3N_28.0} do not depend on $\lambda$.
\end{remark}

We also recall the following elementary inequality (see \cite[p. 122]{stroock}).
\begin{lemma}\label{stroock} For $p>1$, there exists a positive constant $K_p$, depending on $p$, such that 
for all $a,b\in\mathbb{R}$
$$\left\vert\vert b\vert^p-\vert a\vert^p-\vert b-a\vert^p\right\vert\leq K_p\left(\vert b-a\vert^{p-1}\vert a\vert+\vert a\vert^{p-1}\vert b-a\vert \right).$$
\end{lemma}
By adapting the proof of Lemma 6.1 in Miyagaki, Motreanu and Pereira \cite{O}, we obtain the following result. 
Some details are to be found there.

\begin{lemma}\label{C3N_50.0}
Suppose that $N>sp$. For $\varepsilon>0$ small enough, the following estimate is true.
\begin{align*}
I_{\lambda,s}(u)\leq &\;\dfrac{s}{N}\left(\dfrac{1}{b}\right)^{\frac{N-sp}{sp}}\left(\dfrac{\|e_{\varepsilon,\delta}\|_{{X_p^s}}^p-a\|e_{\varepsilon,\delta}\|_{L^p(\Omega)}^p}{\|u_{\varepsilon,\delta}\|_{L^{p_s^*}(\Omega)}^p}\right)^{\frac{N}{sp}}+\dfrac{\lambda}{q}\int_{\Omega}\vert u\vert^{q}\dd x\\
&+\left\{ \begin{array}{ll}
K_0\varepsilon^{s(p-1)}\vert\log(\frac{\varepsilon}{\delta})\vert^{p^*_s}, &\textrm{if}\quad 
p=\frac{2N}{N+s}\,\textrm{ and }\,N>sp^2,\\
K_0\varepsilon^{s(p-1)}, &\textrm{if}\quad 1< p<\frac{2N}{N+s}\,\textrm{ and }\,N>sp^2,\\
K_0\varepsilon^{s(p-1)}+K_0\varepsilon^{s}, &\textrm{if}\quad p> \frac{2N}{N+s}\,\textrm{ and }\,N>sp^2,
\end{array} \right.
\end{align*}
for all $u\in Q_{\varepsilon, R_1, R_2}$, where the constant $K_0>0$ does not depend on $\varepsilon.$
\end{lemma}
\begin{proof}By applying Lemma \ref{stroock} for $a=u_1(x)$ and $b=u_1(x)+re_{\varepsilon,\delta}(x)$, 
the equivalences of norms in finite dimensional spaces guarantees that 
\begin{align*}
&\dfrac{1}{p}\|u_1+re_{\varepsilon,\delta}\|_{X_p^s}^p-\dfrac{a}{p}\|u_1+r e_{\varepsilon,\delta}\|_{L^p(\Omega)}^p 
\leq \; \dfrac{1}{p}\left(\|u_1\|_{X_p^s}^p-a\|u_1\|_{L^p(\Omega)}^p\right)\\
&+\dfrac{r^p}{p}\left(\|e_{\varepsilon,\delta}\|_{X_p^s}^p-a\|e_{\varepsilon,\delta}\|_{L^p(\Omega)}^p\right) +CR_2^{p-1}\|e_{\varepsilon,\delta}\|_{L^p(\Omega)}^{p-1}\|u_1\|_{{X_p^s}}\\
&+CR_2\|u_1\|_{X_p^s}^{p-1}\|e_{\varepsilon,\delta}\|_{L^p(\Omega)}.
\end{align*}

Thus, by applying Lemma \ref{C3N_42} and estimate \eqref{C3N_23}, since $a>\lambda_1$, we obtain
\begin{align*}
I_{\lambda,s}(u)\leq&\dfrac{r^p}{p}\left(\|e_{\varepsilon,\delta}\|_{X_p^s}^p-a\|e_{\varepsilon,\delta}\|_{L^p(\Omega)}^p\right)-\dfrac{b}{p^*_s}\int_{\Omega}(u^+)^{p^*_s}\dd x+\dfrac{\lambda}{q}\int_{\Omega}\vert u\vert^q\dd x\\
&+\left\{ \begin{array}{ll}
2C'\varepsilon^{s(p-1)}\vert\log(\frac{\varepsilon}{\delta})\vert, &\textrm{if}\quad  p=\frac{2N}{N+s}\,\mbox{ and }\,N>sp^2,\\
2C'\varepsilon^{s(p-1)}, &\textrm{if}\quad 1< p<\frac{2N}{N+s}\,\mbox{ and }\,N>sp^2,\\
C'\varepsilon^{s(p-1)}+C'\varepsilon^{s}, &\textrm{if}\quad p> \frac{2N}{N+s}\,\mbox{ and }\,N>sp^2.
\end{array} \right.
\end{align*}

We control the term  $-b/p_s^*\|u^+\|_{L^{p^*_s}(\Omega)}^{p_s^*}$ by applying estimates \eqref{C3N_20} and \eqref{C3N_21} and conclude the existence of a constant $C_6>0$ such that 
\begin{align*}
-\dfrac{b}{p^*_s}\int_{\Omega}\big( u^+\big)^{p^*_s}\dd x
\leq&-\dfrac{br^{p^*_s}}{p_s^*}\int_{\Omega}\big(u_{\varepsilon,\delta}\big)^{p_s^*}\dd x\\& +\left\{ \begin{array}{ll}
C_6 \varepsilon^{\frac{N-sp}{p}}\vert\log(\frac{\varepsilon}{\delta})\vert^{p^*_s}, &\textrm{if}\quad  p=\frac{2N}{N+s},\\
C_6\varepsilon^{\frac{N-sp}{p}}+C_6\varepsilon^{N(p^*_s-1)}, &\textrm{if}\quad 1<p<\frac{2N}{N+s},\\
C_6\varepsilon^{\frac{N-sp}{p}}, &\textrm{if}\quad p>\frac{2N}{N+s}.
	\end{array} \right.
	\end{align*}
Thus, there exists a constant $K_0>0$ such that
\begin{align}\label{C3N_64}
I_{\lambda,s}(u)\leq&\dfrac{r^p}{p}\left(\|e_{\varepsilon,\delta}\|_{X^s_p}^p-a\|e_{\varepsilon,\delta}\|_{L^p(\Omega)}^p\right)-\dfrac{br^{p^*_s}}{p^*_s}\int_{\Omega}\vert u_{\varepsilon,\delta}\vert^{p^*_s}\dd x+\dfrac{\lambda}{q}\int_{\Omega}\vert u\vert^q\dd x\nonumber\\
&+\left\{ \begin{array}{ll}
K_0\varepsilon^{s(p-1)}\vert\log(\frac{\varepsilon}{\delta})\vert^{p^*_s}, &\textrm{if}\quad
p=\frac{2N}{N+s}\,\textrm{ and }\,N>sp^2,\\
K_0\varepsilon^{s(p-1)}, &\textrm{if}\quad 1< p<\frac{2N}{N+s}\,\textrm{ and }\,N>sp^2,\\
K_0\varepsilon^{s(p-1)}+K_0\varepsilon^{s}, &\textrm{if}\quad p> \frac{2N}{N+s}\,\textrm{ and }\,N>sp^2.
\end{array} \right.
\end{align}
	
By applying Lemma \ref{C5P_33.0.0} we have that the function $f:[0, +\infty)\to\mathbb{R}$ given by 
$$f(t)=\dfrac{t^p}{p}\left(\|e_{\varepsilon,\delta}\|_{X^s_p}^p-a\|e_{\varepsilon,\delta}\|_{L^p(\Omega)}^p\right)-\dfrac{bt^{p^*_s}}{p^*_s}\|u_{\varepsilon,\delta}\|_{L^{p^*_s}(\Omega)}^{p^*_s},$$
admits its maximum at a point $t_M$ so that  
$$f(t_M)=\dfrac{s}{N}\left(\dfrac{1}{b}\right)^{\frac{N-sp}{sp}}\left(\dfrac{\|e_{\varepsilon,\delta}\|_{X^s_p}^p-
a\|e_{\varepsilon,\delta}\|_{L^p(\Omega)}^p}{\|u_{\varepsilon,\delta}\|^{p}_{L^{p^*_s}(\Omega)}}\right)^{\frac{N}{sp}}.$$
Our result now follows from \eqref{C3N_64}. 
$\hfill\Box$\end{proof}

Also the next result adapts a similar result obtained in \cite{O}, namely Lemma 6.2.
\begin{lemma}\label{C3N_64.0}
Suppose that $N>sp^2$ and $1<p\leq\dfrac{2N}{N+s}$. Then we have
\begin{equation*}
c_s:=\displaystyle\inf_{h\in\Gamma}\displaystyle\sup_{u\in Q_{\varepsilon,R_1,R_2}}I_{\lambda,s}(h(u))<\dfrac{s}{N}b^\frac{sp-N}{sp}S_{p,s}^\frac{N}{sp}
\end{equation*}
for all $\varepsilon>0$ and $\lambda>0$ small enough,
where $$\Gamma=\left\{h\in C(\overline{Q}_{\varepsilon,R_1,R_2},X^s_p)\;:\;h=id\textrm{ in }\partial Q_{\varepsilon,R_1,R_2}\right\}.$$

The same result is also valid if  $N>sp\left((p-1)^2+p\right)$ and $p>2N/(N+s)$.
\end{lemma}
\begin{proof}
Since $h=id_{Q_{\varepsilon,R_1,R_2}}\in\Gamma$, the compactness of $Q_{\varepsilon,R_1,R_2}$ assures that is enough to prove
\begin{equation}\label{C3N_66}
I_{\lambda,s}(u)<\dfrac{s}{N}b^\frac{sp-N}{sp}S_{p,s}^\frac{N}{sp},\; \forall\, u\in Q_{\varepsilon,R_1,R_2}.
\end{equation}
By applying Lemma \ref{C3N_35} for $a=u_{\varepsilon,\delta}-P_1^s u_{\varepsilon,\delta}=e_{\varepsilon,\delta}$ and $b=P_{2}^{s}u_{\varepsilon,\delta}$, for some constant $c_0>0$ we have, 
\begin{equation*}
\|e_{\varepsilon,\delta}\|_{L^p(\Omega)}^p\geq\frac{1}{c_0}\|u_{\varepsilon,\delta}\|_{L^p(\Omega)}^p-\|P_{1}^{s}u_{\varepsilon,\delta}\|_{L^p(\Omega)}^p.
\end{equation*}
Since $N>sp^2$, it follows from Lemma \ref{a.6} that $\|u_{\varepsilon,\delta}\|^p_{L^p(\Omega)}\geq c_2\varepsilon^{sp}$ for a constant $c_2>0$. Taking into account \eqref{C3N_22}, we obtain 
\begin{equation*}
\|e_{\varepsilon,\delta}\|_{L^p(\Omega)}^p\geq\frac{c_2}{c_0}\varepsilon^{sp}-\left\{ \begin{array}{ll}
c_1\varepsilon^{N}\vert\log(\frac{\varepsilon}{\delta})\vert^p, &\textrm{if}\quad  p=\frac{2N}{N+s}\,\textrm{ and }\,N>sp^2, \\
c_1\varepsilon^{N(p-1)+sp}, &\textrm{if}\quad 1< p<\frac{2N}{N+s}\,\textrm{ and }\,N>sp^2,\\
c_1\varepsilon^{\frac{N-sp}{p-1}}, &\textrm{if}\quad p> \frac{2N}{N+s}\,\textrm{ and }\,N>sp^2.
\end{array} \right. 
\end{equation*}
	
But Lemma \ref{a.6} yields
\begin{equation*}
\|u_{\varepsilon,\delta}\|_{L^{p^*_s}}^{p^*_s}\geq S_{p,s}^{\frac{N}{sp}}+O(\varepsilon^{\frac{N}{p-1}}).
	\end{equation*}
	
Now, mimicking \cite[p.286]{Ch}, we obtain
$$\left\vert \|e_{\varepsilon,\delta}\|_{{X_p^s}}^p-\|u_{\varepsilon,\delta}\|_{{X_p^s}}^p\right\vert \leq c_3\|u_{\varepsilon,\delta}\|_{{X_p^s}}^{p-1}\| P_{1}^{s} u_{\varepsilon,\delta}\|_{L^{\infty}(\Omega)}+c_3\| P_{1}^{s} u_{\varepsilon,\delta}\|_{L^{\infty}(\Omega)}^{p}.$$
	
A new application of Lemma \ref{a.6} and the estimate \eqref{C3N_22} give

\begin{equation*}
\|e_{\varepsilon,\delta}\|_{{X_p^s}}^p\leq\|u_{\varepsilon,\delta}\|_{{X_p^s}}^p+\left\{ \begin{array}{ll}
c_4 \varepsilon^{\frac{N}{p}}\vert\log(\frac{\varepsilon}{\delta})\vert, &\textrm{if}\quad  p=\frac{2N}{N+s}\\
c_4\varepsilon^{N-\frac{N-sp}{p}}, &\textrm{if}\quad 1< p<\frac{2N}{N+s},\\
c_4\varepsilon^{\frac{N-sp}{p(p-1)}}, &\textrm{if}\quad p> \frac{2N}{N+s}.
\end{array} \right. 
\end{equation*}

\textit{Case 1:} Suppose $N>sp^2$ and $p=2N/(N+s)$. 

By applying the mean value theorem to the function  $f(t)=(1+t)^{\frac{N-sp}{N}}$, we conclude that
\begin{equation*}
\left\vert 1-\bigg(1+S_{p,s}^{-\frac{N}{sp}}O(\varepsilon^{\frac{N}{p-1}})\bigg)^{\frac{N-sp}{N}}\right\vert=O(\varepsilon^{\frac{N}{p-1}}).
\end{equation*} 
Then adding and subtracting $\left(1+S_{p,s}^{-\frac{N}{sp}}O(\varepsilon^{\frac{N}{p-1}})\right)^{\frac{N-sp}{N}}S_{p,s}^{\frac{N}{sp}},$ we have 

\begin{align*}
\lefteqn{\frac{\|e_{\varepsilon,\delta}\|_{X_p^s}^p-a\|e_{\varepsilon,\delta}\|_{L^p(\Omega)}^p}{\|u_{\varepsilon,\delta}\|_{L^{p^*_s}(\Omega)}^p}}\\
&\leq &S_{p,s}+  \dfrac{\varepsilon^{sp}\left[ O(\varepsilon^{\frac{N}{p-1}-sp})+ O(\varepsilon^{\frac{N-sp}{p-1}-sp})+\left(c_4+ac_2\right)\varepsilon^{\left(\frac{N}{p}-sp\right)}\vert\log(\frac{\varepsilon}{\delta})\vert-a\dfrac{c_2}{c_0}\right]}{\left(S_{p,s}^{N/sp}+O(\varepsilon^{\frac{N}{p-1}}) \right)^{\frac{N-sp}{N}}},
\end{align*}
from what follows
\begin{equation}\label{Sps}
\frac{\|e_{\varepsilon,\delta}\|_X^p-a\|e_{\varepsilon,\delta}\|_{L^p(\Omega)}^p}{\|u_{\varepsilon,\delta}\|_{L^{p^*_s}(\Omega)}^p}<S_{p,s},
\end{equation}
if $\varepsilon>0$ is small enough. Therefore, 
\begin{equation*}
\dfrac{s}{N}\left(\dfrac{1}{b}\right)^{\frac{N-sp}{sp}}\left(\frac{\|e_{\varepsilon,\delta}\|_{X_p^s}^p-a\|e_{\varepsilon,\delta}\|_{L^p(\Omega)}^p}{\|u_{\varepsilon,\delta}\|_{L^{p^*_s}(\Omega)}^p}\right)^{\frac{N}{sp}}<\dfrac{s}{N}\left(\dfrac{1}{b}\right)^{\frac{N-sp}{sp}}S_{p,s}^{\frac{N}{sp}},
\end{equation*}
The boundness of $Q_{\varepsilon,R_1,R_2}$  and Lemma \ref{C3N_50.0} guarantee that \eqref{C3N_66} is true for $\varepsilon>0$ and $\lambda>0$ small enough.
	
\textit{Case 2:} $N>sp^2$ and $1<p<\dfrac{2N}{N+s}$ or $N>sp\big((p-1)^2+p\big)$ and $\,p>\dfrac{2N}{N+s}$.  
The proofs are analogous to that of \textit{Case 1} . The conclusion follows once obtained the inequality \eqref{Sps}.
$\hfill\Box$\end{proof}\vspace*{.5cm}

\section{\textbf{Proof of Theorem \ref{t0}.}}

\begin{proof} The positive and negative solutions were obtained in Section \ref{Positivenegative}.
Let us denote $u_1$ the positive solution and $u_2$ the negative solution of problem \ref{principal} for every $ \lambda > 0$ sufficiently small. Obviously, these two solutions are distinct. In order to find a third nontrivial solution $u_3$ for the problem \eqref{principal} whose existence depend of the parameter $\lambda$ small enough, we apply the Linking Theorem to the functional $I_{\lambda,s}: X^s_p \rightarrow \mathbb{R}$. Indeed, the geometric conditions of the Linking Theorem follow from Proposition \ref{C2P_25} and Lemma \ref{C3N_28.0}. According to Lemmas \ref{C3P_3} and \ref{C3N_64.0}, $I_{\lambda,s}$ satisfies the $(PS)_{c_s}$-condition, with $$c_s=\displaystyle\inf_{h\in\Gamma}\displaystyle\sup_{u\in Q_{\varepsilon,R_1,R_2}}I_\lambda(h(u))<\dfrac{s}{N}b^\frac{sp-N}{sp}S_{p,s}^\frac{N}{sp}$$ 
and $\Gamma=\{h\in C(\overline{Q}_{\varepsilon,R_1,R_2},X_p^s)\;;\;h=id\textrm{ em }\partial Q_{\varepsilon,R_1,R_2}\}$, if $\varepsilon>0,\lambda>0$ are small enough and $R_1>0,R_2>0$ are large enough. Therefore, the solution $u_3$ is obtained by applying the Linking Theorem.  Since 
\begin{equation}\label{C3N_74}
I_{\lambda,s}(u_1)=I^{+}_{\lambda,s}(u_1):=C^+_{\lambda}\leq\frac{\lambda t_0^q}{q}\int_{\Omega}\varphi_1^q\dd x<\alpha\leq c_s=I_{\lambda,s}(u_3),
\end{equation}
we conclude that $u_1\neq u_3$. By the same reasoning, we conclude that $u_2\neq u_3$. We are done.
$\hfill\Box$

Observe that it follows from \eqref{C3N_74} and the analogous equation for $I^-_{\lambda,s}$ that, in the case 
$0<\lambda<q\alpha/\int_{\Omega}\varphi_1^q\dd x$, then solutions $u_1,u_2$ and $u_3$ are distinct, for $\alpha>0$ given in Proposition \ref{C2P_25}.
\end{proof}

\end{document}